\documentclass[11pt]{amsart}
\usepackage{amsmath, amssymb, amsthm,amsfonts}
\usepackage{enumitem}
\usepackage{graphicx,color}
\usepackage{graphics}
\usepackage{comment}
\usepackage[OT2,T1]{fontenc}
\usepackage{tikz-cd}
\DeclareSymbolFont{cyrletters}{OT2}{wncyr}{m}{n}
\DeclareMathSymbol{\Sha}{\mathalpha}{cyrletters}{"58}

\DeclareMathOperator{\Span}{Span}
\DeclareMathOperator{\SL}{SL}

\DeclareMathOperator{\GL}{GL}

\newcommand{\Hup}{\mathbb{H}}

\newcommand{\A}{\mathbb{A}}

\newcommand{\Q}{{\mathbb Q}}
\newcommand{\Z}{{\mathbb Z}}
\newcommand{\N}{{\mathbb N}}
\newcommand{\C}{{\mathbb C}}
\newcommand{\R}{{\mathbb R}}

\newcommand{\kro}[2]{\left( \frac{#1}{#2} \right) }

\newcommand{\mat}[4]{\left(\begin{matrix} #1 & #2 \\ #3 & #4 \end{matrix}\right)}

            \DeclareFontFamily{U}{wncy}{}
            \DeclareFontShape{U}{wncy}{m}{n}{%
               <5>wncyr5%
               <6>wncyr6%
               <7>wncyr7%
               <8>wncyr8%
               <9>wncyr9%
               <10>wncyr10%
               <11>wncyr10%
               <12>wncyr6%
               <14>wncyr7%
               <17>wncyr8%
               <20>wncyr10%
               <25>wncyr10}{}
\DeclareMathAlphabet{\cyr}{U}{wncy}{m}{n}

\begin{document}

\newtheorem{thm}{Theorem}
\newtheorem*{thms}{Theorem}
\newtheorem{lem}{Lemma}[section]
\newtheorem{prop}[lem]{Proposition}
\newtheorem{alg}[lem]{Algorithm}
\newtheorem{cor}[lem]{Corollary}
\newtheorem{conj}[lem]{Conjecture}
\newtheorem{remark}{Remark}
\theoremstyle{definition}
\newtheorem{ex}{Example}

\title[New forms on $\Gamma_0(m)$]{Hecke algebras, new vectors and new forms on $\Gamma_0(m)$.}
\author{Ehud Moshe Baruch}
\address{Department of Mathematics\\
         Technion\\
         Haifa , 32000\\
         Israel}
\email{embaruch@math.technion.ac.il}
\thanks{}
\author{Soma Purkait}
\address{Faculty of Mathematics\\
         Kyushu University\\
         Japan}
\email{somapurkait@gmail.com}
\date{November 2014}
\begin{abstract}
We characterize the space of new forms for $\Gamma_0(m)$ as a common 
eigenspace of certain Hecke operators which depend on primes $p$ 
dividing the level $m$. To do that we find generators and relations for 
a $p$-adic Hecke algebra of functions on $K=\GL_2(\mathbb{Z}_p)$. 
We explicitly find the $n+1$ irreducible representations of $K$ which 
contain a vector of level $n$ including the unique representation that 
contains the ``new vector'' at level $n$.
After translating the $p$-adic Hecke operators
that we obtain into classical Hecke operators we obtain the results
about the new space mentioned above.
\end{abstract}
\maketitle
\section{Introduction}
The theory of Hecke operators and new forms of integer weight for
$\Gamma_0(m)$ was developed by Atkin and Lehner for the case of
trivial central character \cite{A-L} and by Atkin-Lehner-Li-Miyake 
for arbitrary central characters \cite{Miyake}. Atkin and Lehner define 
Hecke operators $T_q$ for primes $q$ not dividing $m$ and operators $U_p$ 
for primes $p$ dividing $m$. They define the new space of cusp forms on 
$\Gamma_0(m)$ as the space orthogonal under the Petersson inner product
to all the old forms on $\Gamma_0(m)$ which are forms that
come from lower levels $m'$ dividing $m$. They show that all
the Hecke operators stabilize the new space, that they
commute and are diagonalizable. Further, there is a common basis
of eigenforms where each eigenspace is one dimensional and
spanned by the primitive eigenforms, the ones whose
first Fourier coefficient is one. A basic tool in the discussion
is a certain involution on the whole space called the Atkin-Lehner involution.
Atkin and Lehner remark that the definition of the new space as
an orthogonal complement does not give enough information on this
space. In this paper we will show how to characterize the new
space using eigenvalues of Hecke operators. In particular when
the primes $p$ divides $m$ or $p^2$ divides $m$ but $p^3$ does
not divide $m$ we will use a certain product of the Atkin-Lehner
involution and the operator $U_p$. When $p^3$ divides $m$, the
information on the new space can not be obtained using the
operators considered by Atkin and Lehner and we will introduce
a family of Hecke operators which "capture" the various spaces
of old forms on $\Gamma_0(m)$. 

The theory of new forms was given a representation theoretic
interpretation by Casselman~\cite{Casselman} \cite{Casselman2} who showed that
every irreducible admissible representation of $\GL_2(F)$ where
$F$ is a $p$-adic field contains a unique new vector. Schmidt~\cite{Schmidt}
used the classification of irreducible admissible representations
of $\GL_2(F)$ to describe the new vectors in these representations.
In their remarkable work on the space of half integral weight modular forms 
Niwa~\cite{Niwa} and Kohnen~\cite{Kohnen} considered a certain
Hecke operator $Q$ which is a composition of classical Hecke operators.  
Kohnen defined the plus space to be a particular eigenspace of this operator. 
Loke and Savin \cite{L-S} interpreted Kohnen's definition in the context of a 
Hecke algeba for the double cover of
$\SL_2(\mathbb{Q}_2)$ and used this Hecke algebra to classify
the representations that
contain maximal level vectors fixed by a certain congruence subgroup.
Using similar methods we will study a Hecke algebra of functions
on $K=\GL_2(\mathbb{Z}_p)$ which are compactly supported and
bi-invariant with respect to an open compact subgroup
$K_0(p^n)$ which is defined below. We will find generators and
relations for this Hecke algebra and show that it is commutative.
We will study the finite dimensional representations of $K$
containing a $K_0(p^n)$ fixed vector. Casselman showed 
that there is a unique irreducible representation of
$K$ which contains a $K_0(p^n)$ fixed vector but
does not contain a $K_0(p^k)$ fixed vector for $k<n$. Such vectors 
are called new vectors. We will explicitly describe these new vectors 
and action of Hecke algebra on such vectors. Using our Hecke algebras 
we will construct classical Hecke operators that are needed to 
classify the new space. We view our paper as a connection between 
the theory of new vectors described by Casselman and the theory of 
newforms by Atkin and Lehner.

\section{The main results}\label{sec:mainresults}
Let $S_{2k}(\Gamma_0(m))$ be the space of cusp forms of weight $2k$ on
$\Gamma_0(m)$. The space of old forms $S_{2k}^{\text{old}}(\Gamma_0(m))$ 
is defined to be the space spanned by all the forms $f(lz)$ where 
$f\in S_{2k}(\Gamma_0(m_1))$ where 
$l,\ m_1\in \mathbb{N}$, with $l m_1|m$ and $m_1\not= m$. The space of
new forms $S_{2k}^{\text{new}}(\Gamma_0(m))$ is the space orthogonal 
to the space of old
forms under the Petersson inner product. 
Let $\GL_2(\mathbb{R})^{+}$ be the group of
$2\times 2$ real matrices with positive determinant and 
$\Hup$ be the upper half plane. 
For $g=\mat{a}{b}{c}{d} \in \GL_2(\mathbb{R})^{+}$ and $z\in \Hup$ define
\[
j(g,z)=det(g)^{-1/2}(cz+d),\]
and for functions $f$ on $\Hup$ define the slash operator $|_{2k}g$ by
\[f|_{2k}g = j(g,z)^{-2k}f\left(\frac{az+b}{cz+d}\right).\]
Let $p$ be a prime dividing $m$. Assume that $p^{n}|m$ and $p^{n+1}\nmid m$. 
(We will denote this by $p^n\|m$.) We define the following operators:
\begin{align*}
\tilde{U}_p(f)(z) &= p^{-k} \sum_{s=0}^{p-1} f((z+s)/p)\\
W_{p^n}(f)(z) &= f|_{2k}\mat{p^n\beta}{1}{m\gamma}{p^n}(z) \quad \text{where $p^{2n}\beta-m\gamma=p^n$}.
\end{align*}
Let $m=p^n m'$ with $p\nmid m'$ and $n \ge 2$. 
We fix $j$ such that $1\leq j \leq n-1$. Let
\[
L_{j}(f)=\sum_{s \in (\Z/ p^{n-j}\Z)^{*}} f|_{2k}A_s
\]
where $A_s \in \SL_2(\Z)$ is any matrix of the form 
$\mat{a_s}{b_s}{p^{j}m'}{p^{n-j}-sm'}$.
In this case we define for $1\leq r \leq n-1$ the operators
\[
S_{p^n,r} = I + \sum_{j=r}^{n-1}L_j
\]
We also define
\[
S_{p^n,r}'=W_{p^n}S_{p^n,r}W_{p^n}^{-1}.
\]
\begin{remark}
The operator $\tilde{U}_p$ is denoted by $U_p^{*}=p^{1-k}U_p$ in Atkin and Lehner
(See \cite{A-L} Lemma 14) where $U_p$ is the usual Hecke operator, sometime also
denoted as $T_p$ (\cite{Miyake}). The operator $W_{p^n}$ is the usual 
Atkin-Lehner involution $W_p$ defined in (\cite{A-L} (2.3)). 
The operators $S_{p^n,r}$ did not
appear in \cite{A-L}.
\end{remark}

Our main theorems characterize the space of new forms as a common eigenspace of 
above defined operators:
\begin{thm}\label{thm:sec2thm1}
Let $N$ be a square-free positive number. 
For any prime $p \mid N$, let $Q_p = \tilde{U}_p W_{p}$ 
and $Q_p' = W_{p}\tilde{U}_p$.
Then the space of new forms $S_{2k}^{\mathrm{new}}(\Gamma_0(N))$ is 
the intersection of
the $-1$ eigenspaces of $Q_{p}$ and $Q_{p}'$ as $p$ varies over the prime 
divisors of $N$. That is, $f \in S_{2k}^{\mathrm{new}}(\Gamma_0(N))$
if and only if $Q_{p}(f)=-f =Q_{p}'(f)$ for all primes $p\mid N$.
\end{thm}
\begin{thm}\label{thm:sec2thm2}
Let $N=M_1^2M$ where $M_1$ and $M$ are square free and coprime. 
For any prime $p$ 
dividing $M_1$, let $Q_{p^2} = (\tilde{U}_p)^2 W_{p^2}$ 
and $Q_{p^2}' = W_{p^2}(\tilde{U}_p)^2$
Then $f \in S_{2k}^{\mathrm{new}}(\Gamma_0(N))$
if and only if $Q_{p}(f)=-f =Q_{p}'(f)$ for all primes $p$ dividing $M$ and 
$Q_{p^2}(f)=0 = Q_{p^2}'(f)$ for all primes $p$ dividing $M_1$.
\end{thm}
\noindent {\bf Theorem 2'.}
Let $N$ be as in Theorem~\ref{thm:sec2thm2}. Then 
$f \in S_{2k}^{\mathrm{new}}(\Gamma_0(N))$
if and only if $Q_{p}(f)=-f =Q_{p}'(f)$ for all primes $p$ dividing $M$ and 
$S_{p^2,1}(f)=0 = S_{p^{2},1}'(f)$ for all primes $p$ dividing $M_1$.
\begin{thm}\label{thm:sec2thm4}
Let $N$ be a positive integer. 
Then the space of new forms $S_{2k}^{\mathrm{new}}(\Gamma_0(N))$ is the 
intersection of the $-1$ eigenspaces of 
$Q_{p}$ and $Q_{p}'$ where $p$ varies over the primes 
such that $p\|N$ and the $0$ eigenspaces of $S_{p^{\gamma},{\gamma-1}}$ and 
$S_{p^{\gamma},{\gamma-1}}'$ for primes $p$  
such that $p^\gamma \|N$ with $\gamma \ge 2$. That is, $f \in S_{2k}^{\mathrm{new}}(\Gamma_0(N))$
if and only if $Q_{p}(f)=-f =Q_{p}'(f)$ for all primes $p$ such that 
$p\| N$ and 
$S_{p^{\gamma},{\gamma-1}}(f) = 0 =S_{p^{\gamma},{\gamma-1}}'(f)$ 
for all primes $p$ such that $p^\gamma \| N$ for $\gamma \ge 2$.
\end{thm}

Let $q=e^{2\pi i z}$ and
$f(z) = \sum_{n=1}^{\infty}a_nq^n \in S_{2k}(\Gamma_0(m))$. Let $p$ be an odd prime.
Define
\[R_p (f)(z) = \sum_{n=1}^{\infty} \kro{n}{p}a_nq^n, 
\qquad R_{\chi} (f)(z) = \sum_{n=1}^{\infty} \kro{-1}{n}a_nq^n.\]
By \cite[Lemma 33]{A-L}, $R_p$ and $R_{\chi}$ are 
operators on $S_{2k}(\Gamma_0(m))$ provided that $p^2 \mid m$ and $16 \mid m$ respectively.
\begin{thm}\label{thm:sec2thm5}
Let $N=2^{\beta}M_1M_2$ where $M_1M_2$ is odd such that 
$M_1$ is square free and any prime divisor of $M_2$ divides 
it with a power at least $2$. Let $\beta \ge 4$. 
Then $ f \in S_{2k}^{\mathrm{new}}(\Gamma_0(N))$ if and only if 
$Q_p(f) = -f = Q_p'(f)$ for all primes $p$ dividing $M_1$, 
$(R_{\chi})^2(f) = f$ and $(R_{p})^2(f) = f$ for all primes $p$ dividing $M_2$, 
and $S_{p^\gamma,\gamma-1}(f)=0$ for all primes $p$ such that 
$p^\gamma \| 2^\beta M_2$.
\end{thm}

\section{ $p$-adic Hecke Algebras and the representations of $K$.}
In this section we will find generators and relations for a Hecke algebra of functions
on $K=\GL_2(\Z_p)$ which are bi-invariant with respect to $K_0(p^n)$. We will use these results to
classify smooth irreducible finite dimensional representations of $K$ which have $K_0(p^n)$ fixed vectors.

Denote by $G$ the group $\GL_2(\Q_p)$. Let $K_0(p^n)$ be the subgroup of $K$ defined by
\[K_0(p^n) =\left\{ \mat{a}{b}{c}{d} \in K\ :\ c \in p^n\Z_p \right\}. \]
The subgroup $K_0(p)$ denotes the usual Iwahori subgroup. In this section
we shall consider the Hecke algebra of $G$ with respect to $K_0(p^n)$.

It is well known that the space $C_c^{\infty}(G)$, the space of locally constant, compactly supported
complex-valued functions on $G$, forms a $\C$-algebra under convolution which, for any
$f_1, f_2 \in C_c^{\infty}(G)$, is defined by
\[ f_1*f_2(h) = \int_G f_1(g)f_2(g^{-1}h) dg = \int_G f_1(hg)f_2(g^{-1}) dg, \]
where $dg$ is the Haar measure on $G$ such that the measure of $K_0(p^n)$ is one.
The Hecke algebra corresponding to $K_0(p^n)$,
denoted by $H(G//K_0(p^n))$, is
the subalgebra of $C_c^{\infty}(G)$ consisting of $K_0(p^n)$ bi-invariant functions:
\[H(G//K_0(p^n)) = \{ f \in C_c^{\infty}(G) : f(kgk')=f(g) \text{ for } g \in G,\ k,\ k' \in K_0(p^n)\}.\]
Let $X_g$ denotes the characteristic function of the double coset $K_0(p^n)gK_0(p^n)$. Then
$H(G//K_0(p^n))$ as a $\C$-vector space is spanned by $X_g$ as $g$ varies over the double coset representatives of
$G$ modulo $K_0(p^n)$.

Let $\mu(K_0(p^n)gK_0(p^n))$ denotes the number of disjoint left 
(right) $K_0(p^n)$ cosets in the
double coset $K_0(p^n)gK_0(p^n)$. Then the following lemmas are well known \cite[Corollary 1.1]{Howe}.
\begin{lem} \label{lem:rel1}
If $\mu(K_0(p^n)gK_0(p^n))\mu(K_0(p^n)hK_0(p^n)) = \mu(K_0(p^n)ghK_0(p^n))$ then $X_g*X_h = X_{gh}$.
\end{lem}

\begin{lem} \label{lem:rel2}
Let $f_1,\ f_2 \in H(G//K_0(p^n)$ such that $f_1$ is supported on
$K_{0}(p^n)xK_{0}(p^n)=\bigcup_{i=1}^{m}\alpha_i K_0(p^n)$ and 
 $f_2$ is supported on $K_{0}(p^n)yK_{0}(p^n)=\bigcup_{j=1}^{n}\beta_j K_0(p^n)$. Then
\[
f_1*f_2(h)=
\sum_{i=1}^{m}f_1(\alpha_i)f_2(\alpha_{i}^{-1}h)\]
where the nonzero summands are precisely for those $i$ for which there exist a $j$ such that
$h\in \alpha_i \beta_j K_0(p^n)$.
\end{lem}

For $t \in \Q_p$ we shall consider the following elements:
\[x(t) =\mat{1}{t}{0}{1},\ y(t)=\mat{1}{0}{t}{1},\  w(t) = \mat{0}{-1}{t}{0},\]
\[d(t)=\mat{t}{0}{0}{1},\ z(t)=\mat{t}{0}{0}{t}.\]
Let $N=\{x(t) : t \in \Q_p \}$, $\bar{N}=\{y(t) : t \in \Q_p \}$ and $A$ be the group of
diagonal matrices of $G$. Let $Z_G = \{z(t) : t \in \Q_p^{*} \}$ denote the center of $G$.

\subsection{The Iwahori Hecke Algebra}
\begin{lem} \label{lem:gen1}
A complete set of representatives for the double cosets of $G$ mod $K_0(p)$
are given by $d(p^n)z(m),\ w(p^n)z(m)$ where $n$, $m$ varies over integers.
\end{lem}
\begin{proof}
For proof refer to \cite[Section 2.3]{Howe}.
\end{proof}
\begin{lem}
\begin{enumerate} \label{lem:rel3}
 \item  For $n\geq 0$ we have
\[
K_0(p) d(p^n) K_0(p) = \bigsqcup_{s\in \Z_p / p^{n}\Z_p}
x(s)d(p^n)K_0(p)=\bigsqcup_{s\in \Z_p / p^{n}\Z_p}K_{0}(p)d(p^n)y(ps).
\]
\item  For $n \geq 1$ we have
\[
K_0(p) d(p^{-n}) K_0(p) = \bigsqcup_{s\in \Z_p / p^{n}\Z_p}
y(ps)d(p^{-n})K_0(p)=\bigsqcup_{s\in \Z_p / p^{n}\Z_p}K_{0}(p)d(p^{-n})x(s).
\]
\item For $n\geq 1$ we have
\[
K_0(p) w(p^n) K_0(p) = \bigsqcup_{s\in \Z_p / p^{n-1}\Z_p}
y(ps)w(p^n)K_0(p)=\bigsqcup_{s\in \Z_p / p^{n-1}\Z_p}K_{0}(p)w(p^n)y(ps).
\]
\item For $n \geq 0$ we have
\[
K_0(p) w(p^{-n}) K_0(p) = \bigsqcup_{s\in \Z_p / p^{n+1}\Z_p}
x(s)w(p^{-n})K_0(p)=\bigsqcup_{s\in \Z_p / p^{n+1}\Z_p}K_{0}(p)w(p^{-n})x(s).
\]
\end{enumerate}
\end{lem}
\begin{proof}
The proof easily follows from the triangular decomposition
\[K_0(p) = (N \cap K_0(p))(A \cap K_0(p))(\bar{N} \cap K_0(p)).\]
\end{proof}

Let $\mathcal{T}_n=X_{d(p^n)}$, $\mathcal{U}_n=X_{w(p^n)}$ and $\mathcal{Z}=X_{z(p)}$ be 
elements of the Hecke algebra
$H(G//K_0(p))$. It is easy to see that $\mathcal{Z}$ commutes with every $f\in H(G//K_0(p))$
and that $\mathcal{Z}^n=X_{z(p^n)}$. We have the following well known lemma.
\begin{lem}\label{lem:rel4}
\begin{enumerate}
 \item If $n,m \geq 0$ or $n,m \leq 0$, then $\mathcal{T}_n * \mathcal{T}_m= \mathcal{T}_{n+m}$.
 \item If $n \geq 0$ then $\mathcal{U}_{1}*\mathcal{T}_n= \mathcal{U}_{n+1}$ and $\mathcal{T}_n*\mathcal{U}_1=\mathcal{Z}^n * \mathcal{U}_{1-n}$.
 \item If $n \geq 0$ then $\mathcal{U}_{1}*\mathcal{T}_{-n}= \mathcal{U}_{1-n}$ and $\mathcal{T}_{-n}*\mathcal{U}_1=\mathcal{Z}^{-n}*\mathcal{U}_{1+n}$.
 \item If $n \geq 0$ then $\mathcal{U}_0*\mathcal{T}_{-n}=\mathcal{U}_{-n}$ and $\mathcal{T}_n*\mathcal{U}_0=\mathcal{Z}^n * \mathcal{U}_{-n}$.
 \item For $n\in \Z$, $\mathcal{U}_{1}*\mathcal{U}_{n}= \mathcal{Z}*\mathcal{T}_{n-1}$ and  $\mathcal{U}_{n}*\mathcal{U}_1=\mathcal{Z}^{n}*\mathcal{T}_{1-n}$.
 \item For $n \geq 1$, $\mathcal{U}_0*\mathcal{U}_n=\mathcal{T}_n$ and $\mathcal{U}_n*\mathcal{U}_0=\mathcal{Z}^n * \mathcal{T}_{-n}$.
 \item $\mathcal{U}_0 *\mathcal{U}_0=(p-1)\mathcal{U}_0+p$
\end{enumerate}
\end{lem}
\begin{proof}
The parts $(1)$ to $(6)$ follows from Lemma~\ref{lem:rel1} and~\ref{lem:rel3}. 

Using Lemma~\ref{lem:rel2} it is easy to see that
$\mathcal{U}_0*\mathcal{U}_0$ is supported only on the double cosets $K_0(p)$ and $K_0(p)w(1)K_0(p)$,
so to obtain $(7)$ enough to find the values of $\mathcal{U}_0*\mathcal{U}_0$ on the elements $w(1)$ and on $1$.
Using Lemma~\ref{lem:rel2} and~\ref{lem:rel3},
\[\mathcal{U}_0 *\mathcal{U}_0(w(1))=\sum_{s=0}^{p-1}\mathcal{U}_0(x(s)w(1))\mathcal{U}_0(w(1)x(-s)w(1))=\sum_{s=0}^{p-1}\mathcal{U}_0(y(-s))\]
For each $1 \le s \le p-1$ we have $y(-s)\in K_0 w(1) K_0$ while clearly $y(0) \not\in K_0 w(1) K_0$,
hence $\mathcal{U}_0*\mathcal{U}_0(w(1))=p-1$. Further,
\[
\mathcal{U}_0 *\mathcal{U}_0(1)=\sum_{s=0}^{p-1}\mathcal{U}_0(x(s)w)\mathcal{U}_0(wx(-s))=\sum_{s=0}^{p-1}\mathcal{U}_0(w)=p.
\]
\end{proof}
Thus we obtain the following well known theorem:
\begin{thm}\label{thm:thm1Hecke}
The Iwahori Hecke Algebra $H(G//K_0(p))$ is generated by $\mathcal{U}_0$, 
$\mathcal{U}_1$ and $\mathcal{Z}$
with the relations:\\
1) $\mathcal{U}_1^{2}=\mathcal{Z}$\\
2) $(\mathcal{U}_0-p)(\mathcal{U}_0+1)=0$\\
3) $\mathcal{Z}$ commutes with $\mathcal{U}_0$ and $\mathcal{U}_1$
\end{thm}
\begin{remark}
The algebra $H(G//K_0(p)) / \langle \mathcal{Z} \rangle $ is an algebra generated by $\mathcal{U}_0$ and $\mathcal{U}_1$ with the relations
$\mathcal{U}_1^{2}=1$ and $(\mathcal{U}_0-p)(\mathcal{U}_0+1)=0$.
\end{remark}

\subsection{A subalgebra}

Let $H(K//K_0(p^n))$ denotes the subalgebra of the algebra $H(G//K_0(p^n))$ consisting of functions supported on $K$.
We shall now be looking at generators and relations for $H(K//K_0(p^n))$ when $n \geq 2$.

We consider the double cosets of $K$ mod $K_0(p^n)$.
We first note the following lemma~\cite[Lemma 1]{Casselman2}.
\begin{lem} \label{L:reps2}
A complete set of representatives for the double cosets of $K$ mod $K_0(p^n)$ are given by
$1,\ w(1),\ y(p),\ y(p^2),\ \ldots\ y(p^{n-1})$.
\end{lem}

For simplicity, we shall write $K_0$ for $K_0(p^n)$. 

Let $\mathcal{U}_0=X_{w(1)}$ and \mbox{$\mathcal{V}_{r} = X_{y(p^r)}$} for $1 \le r \le n-1$ be the elements of
$H(G//K_0)$. Then by the above lemma, $H(K//K_0)$ is spanned by $1$, $\mathcal{U}_0$ and
$\mathcal{V}_{r}$ where $1 \le r \le n-1$.

We shall need the following lemmas.
\begin{lem}\label{lem:lem1p^r}
Assume that $r$ satisfies $n>r\geq n/2$. Then
\[
K_0y(p^r)K_0=\bigsqcup_{s \in \Z_p^{*}/1+p^{n-r}\Z_p}
d(s)y(p^{r})K_{0}
=\bigsqcup_{s \in \Z_p^{*}/1+p^{n-r}\Z_p}
K_{0}y(p^{r})d(s)
\]
\end{lem}
\begin{proof}
Since $K_0 = N'A'\bar{N}'$ where $N'=N \cap K_0$, $A'=A \cap K_0$ and $\bar{N}' =\bar{N} \cap K_0$,
and  $A' = DZ'$ where $D$ consists of matrices $d(a) \in K$ and $Z' = Z_G \cap K$, we have
\[K_0y(p^r)K_0 =  N'A'\bar{N}'y(p^r)K_0 = N'A'y(p^r)K_0 = N'Dy(p^r)K_0 .\]
Now any $a \in \Z_p^{*}$ can be written as $a =sa'$ where 
$a' \in 1+p^{n-r}\Z_p$ and $s\in \Z_p^{*}/1+p^{n-r}\Z_p$. Since
\[y(-p^r)d(a')y(p^r) = \mat{a'}{0}{p^r(1-a')}{1} \in K_0\]
we get that
\[K_0y(p^r)K_0 = \bigcup_{s \in \Z_p^{*}/1+p^{n-r}\Z_p}N'd(s)y(p^{r})K_{0}.\]
We obtain the decomposition since
\[N'd(s)=d(s)N' \quad \text{and} \quad y(-p^r)x(u)y(p^r)=\mat{1+up^{r}}{u}{-up^{2r}}{1-up^r} \in K_0. \]
Now we show that the union is disjoint.
Let $g_1=d(s_1)y(p^{r})$ and $g_2=d(s_2)y(p^{r})$. Assume
$g_1^{-1}g_2 \in K_0$ then
\[
y(-p^r)d(s_1^{-1}s_2)y(p^r)= \mat{s_1^{-1}s_2}{0}{(1-s_1^{-1}s_2)p^r}{1} \in  K_0,
\]
hence $s_1^{-1}s_2\in 1+p^{n-r}\Z_p$.
\end{proof}
\begin{lem} \label{L:stab}
Assume that $0<r<n/2$. Let $K_0^{y(p^r)}=y(p^r)K_0y(p^r)^{-1}\cap K_0$. Then an element of $K_0^{y(p^r)}$ 
can be written as $y(v)z(t)d(s)x(u)$ where $v\in p^{n}\Z_p$, $t,\ s\in \Z_p^{*}$, $u\in \Z_p$ and 
$s-1-p^r u\in p^{n-r}\Z_p$.
\end{lem}
\begin{lem}\label{lem:lem2p^r}
Assume that $r$ satisfies $0<r< n/2$. Then
\[
K_0y(p^r)K_0=\bigsqcup_{s \in \Z_p^{*}/1+p^{n-r}\Z_p}
d(s)y(p^{r})K_{0}
=\bigsqcup_{s \in \Z_p^{*}/1+p^{n-r}\Z_p}
K_{0}y(p^{r})d(s)
\]
\end{lem}
\begin{proof}
As in Lemma~\ref{lem:lem1p^r} we can write
$g=d(s)x(u)y(p^r)k_0$ where $s\in \Z_p^{*}$, $u\in \Z_p$ and $k_0\in K_0$. Now
\[
g=d(s)d(1+p^{r}u)^{-1}d(1+p^{r}u)x(u)y(p^r)k_0
\]
It follows from Lemma~\ref{L:stab} that $d(1+p^{r}u)x(u)\in K_0^{y(p^r)}$. Let $s_1=s(1+p^{r}u)^{-1} \in \Z_p^*$. Then
we get that $g=d(s_1)y(p^r)k_1$
for some $k_1\in K_0$ hence we get the decomposition as in the statement . The disjointness follows as in Lemma~\ref{lem:lem1p^r}.
\end{proof}

\begin{prop}\label{prop:reln1}
We have the following relations in $H(K//K_0)$:
\begin{enumerate}
 \item $\mathcal{V}_r^2 = p^{n-r-1}(p-1) ( I + \sum_{j=r+1}^{n-1} 
 \mathcal{V}_{j} ) + p^{n-r-1}(p-2) \mathcal{V}_r$.
 \item $\mathcal{V}_r* \mathcal{V}_{j} = (p-1)p^{n-j-1}\mathcal{V}_r = \mathcal{V}_{j}* \mathcal{V}_r \quad \mathrm{ for } \quad r+1 \le j \le n-1.$
 \item Let $\mathcal{Y}_{r+1} = I + \sum_{j=r+1}^{n-1} \mathcal{V}_{j}$. Then
 \[\mathcal{V}_r * \mathcal{Y}_{r+1} = p^{n-r-1} \mathcal{V}_r = \mathcal{Y}_{r+1} * \mathcal{V}_r,\] and so,
 \[(\mathcal{V}_r - p^{n-r-1}(p-1))(\mathcal{V}_r +\mathcal{Y}_{r+1})=0.\]
\end{enumerate}
\end{prop}
\begin{proof}
For $(1)$, we first compute the support of $\mathcal{V}_r*\mathcal{V}_r$. By Lemma~\ref{lem:lem1p^r} and~\ref{lem:lem2p^r},
\[K_0y(p^r)K_0 = \bigsqcup_{s \in \Z_p^*/ 1+p^{n-r}\Z_p} \alpha_sK_0 \quad \text{ where $\alpha_s= d(s)y(p^r)$},\]
so using Lemma~\ref{lem:rel2} we get that
$\mathcal{V}_r*\mathcal{V}_r$ is supported on those $g \in G$ for which there exists $s,\ t \in \Z_p^*/ 1+p^{n-r}\Z_p$
such that
\[(\alpha_s\alpha_t)^{-1}g = \mat{\frac{1}{st}}{0}{\frac{-p^r(t+1)}{st}}{1}g \in K_0.\]
Clearly it is enough
to check the support on $g =1,\ w(1),\ y(p^{j})$ for $1 \le j \le n-1$.
Note that $(\alpha_s\alpha_t)^{-1}w(1) = \mat{0}{*}{1}{*} \not\in K_0$.
For $g=1$ taking $s=1$ and $t=p^{n-r}-1 \in \Z_p^*/ 1+p^{n-r}\Z_p$ we get that $\mathcal{V}_r*\mathcal{V}_r$ is supported on $K_0$. For $g=y(p^{j})$,
\[(\alpha_s\alpha_t)^{-1}g \in K_0 \iff p^{j}st-p^r(t+1) \in p^n\Z_p.\]
If $j < r$, this is impossible. First assume that $r<j <n$, then the above equation holds if and only if
$p^{j-r}st - (t+1) \in p^{n-r}\Z_p$. Taking $t=p^{j-r}-1$ and $s = (1+ p^{n-j})t^{-1} \in \Z_p^*/ 1+p^{n-r}\Z_p$, we are done.
Now assume $j= r$. If $p>2$ then taking $t= p^{n-r}-2$ and $s=-1/t$ we are done, if $p=2$ no choice of $s,\ t$ works.
Thus we get that $\mathcal{V}_r*\mathcal{V}_r$ is supported on $K_0$ and $K_0y(p^j)K_0$ where if $p>2$ then $r \le j < n$ while for $p=2$ we have $r <j < n$.
Since $y(-p^r) \in K_0y(p^r)K_0$,
\[\mathcal{V}_r*\mathcal{V}_r(1) = \sum_{s \in \Z_p^*/ 1+p^{n-r}\Z_p}\mathcal{V}_r(y(-p^r)) = p^{n-r-1}(p-1).\]

For $r \le j < n$,
\[\mathcal{V}_r*\mathcal{V}_r(y(p^{j})) =\sum_{s \in \Z_p^*/ 1+p^{n-r}\Z_p}\mathcal{V}_r(y(-p^r)d(s)y(p^{j})).\]
We want to check for which $s$, there exists a matrix $A=\mat{a}{b}{c}{d} \in K_0$ such that
$y(-p^r)d(s)y(p^{j})Ay(-p^r) \in K_0 \text{ i.e., } (p^{j-r}-s^{-1})(a-bp^r)-d \in p^{n-r}\Z_p$.
If $r < j$ then for any $s \in \Z_p^*$ take $b=c=0$, $a= \frac{p^{n-r}-1}{p^{j-r}-s^{-1}},\ d=-1$, thus
$\mathcal{V}_r * \mathcal{V}_r(y(p^{j})) = p^{n-r-1}(p-1)$. If $p>2$ and $j=r$, it is easy to see that such an $A$ exists if and only if $s \not\in 1+p\Z_p$,
in this case take $b=c=0$ and $a= \frac{p^{n-r}-1}{1-s^{-1}},\ d=-1$. The number of $s \in \Z_p^*/ 1+p^{n-r}\Z_p$ such that
$s \not\in 1+p\Z_p$ is equal to $p^{n-r-1}(p-2)$ and so $\mathcal{V}_r*\mathcal{V}_r(y(p^{r})) = p^{n-r-1}(p-2)$.

For $(2)$, as before for $r+1 \le j<n$, we get that $\mathcal{V}_r* \mathcal{V}_{j}$ is supported at $g \in G$ if and only if there exists $s \in \Z_p^*/ 1+p^{n-r}\Z_p$ and
$t \in \Z_p^*/ 1+p^{n-j}\Z_p$ such that
\[\mat{\frac{1}{st}}{0}{\frac{-(p^rt+p^{j})}{st}}{1}g \in K_0.\]
It is easy to check that the above does not hold for $g=1,\ w(1),\ y(p^{i})$ for $i \ne r$. If $i=r$, taking $s=p^{j -r}+1$, $t=1$ we are done.
Similarly $\mathcal{V}_{j}* \mathcal{V}_r$ is supported only on $K_0y(p^r)K_0$.
Now
\[\mathcal{V}_r*\mathcal{V}_{j}(y(p^r)) =\sum_{s \in \Z_p^*/ 1+p^{n-r}\Z_p}\mathcal{V}_{j}(y(-p^r)d(s^{-1})y(p^r)),\]
so we want to count $s$, for which there exists $A=\mat{a}{b}{c}{d} \in K_0$ such that
$y(-p^r)d(s^{-1})y(p^r)Ay(-p^{j}) \in K_0 \text{ i.e., }(1-s^{-1})(a-bp^{j})-dp^{j-r} \in p^{n-r}\Z_p$,
which holds if and only if $s-1 \in p^{j-r}\Z_p^*$, in which case if $s-1= p^{j-r}u$ then taking $b=0$, $a=s$, $d=u$ we are done. Thus
$\mathcal{V}_r*\mathcal{V}_{j} = C_{j}\mathcal{V}_r$ where for $r+1 \le j < n$,
\[C_j = \#\{s \in \Z_p^*/ 1+p^{n-r}\Z_p : s-1 \in p^{j-r}\Z_p^*\}= (p-1)p^{n-j-1}.\]
For $\mathcal{V}_{j}*\mathcal{V}_{r}(y(p^r))$ we use that $K_0y(-p^r)K_0 = \bigsqcup_{s \in \Z_p^*/ 1+p^{n-r}\Z_p} d(s)y(-p^r)K_0$ to get
\[\mathcal{V}_{j}*\mathcal{V}_{r}(y(p^r))=\sum_{s \in \Z_p^*/ 1+p^{n-r}\Z_p}\mathcal{V}_{j}(y(p^r)d(s)y(-p^r)),\]
the calculations now follow as above.

For $(3)$,
\[\mathcal{V}_r* \mathcal{Y}_{r+1} = \mathcal{V}_r + (p-1)\mathcal{V}_r + (p-1)p\mathcal{V}_r + \cdots +(p-1)p^{n-r-2}\mathcal{V}_r \]
\[= \mathcal{V}_r + (p^{n-r-1} - 1)\mathcal{V}_r = p^{n-r-1} \mathcal{V}_r,\] the rest follows from $(1)$.
\end{proof}

For $1\le r\le n-1$, let $\mathcal{Y}_r$ be as before, i.e. $\mathcal{Y}_r = I + \sum_{j=r}^{n-1}\mathcal{V}_j$, take $\mathcal{Y}_n=I$. We have following
easy corollary.
\begin{cor}\label{cor:reln1}
\begin{enumerate}
 \item $\mathcal{Y}_{n-r}^2 =p^{r}\mathcal{Y}_{n-r}$ for all $0\le r\le n-1$.
 \item $\mathcal{Y}_r *\mathcal{Y}_l = p^{n-r}\mathcal{Y}_{l} = \mathcal{Y}_l*\mathcal{Y}_r$ for $r\ge l$.
\end{enumerate}
\end{cor}
\begin{proof}
Note that $\mathcal{V}_{n-r}=\mathcal{Y}_{n-r}-\mathcal{Y}_{n-r+1}$ for all $1\le r\le n-1$.
Clearly $(1)$ holds for $r=0$. Assume that $\mathcal{Y}_{n-(a-1)}^2 =p^{a-1}\mathcal{Y}_{n-(a-1)}$. Then
using lemma~\ref{prop:reln1}
\begin{equation*}
\begin{split}
& \quad \mathcal{Y}_{n-a}^2 = (\mathcal{Y}_{n-(a-1)}+\mathcal{V}_{n-a})(\mathcal{Y}_{n-(a-1)}+\mathcal{V}_{n-a})\\
&= \mathcal{Y}_{n-(a-1)}^2 + 2\mathcal{Y}_{n-(a-1)}\mathcal{V}_{n-a}+ \mathcal{V}_{n-a}^2\\
&=p^{a-1}\mathcal{Y}_{n-(a-1)}+ 2p^{a-1}\mathcal{V}_{n-a} + (p-1)p^{a-1}\mathcal{Y}_{n-(a-1)} +(p-2)p^{a-1}\mathcal{V}_{n-a}\\
&=p^{a}\mathcal{Y}_{n-(a-1)}+ p^{a}\mathcal{V}_{n-(a-1)}\\
&=p^{a}\mathcal{Y}_{n-a}.
\end{split}
\end{equation*}
Similarly for $(2)$, let $r=l+m$ for some $m\ge 0$. Then
\[\mathcal{Y}_r*\mathcal{Y}_l = \mathcal{Y}_r*(\mathcal{V}_l + \mathcal{V}_{l+1}+\mathcal{V}_{l+2}+ \cdots \mathcal{V}_{l+m-1} +\mathcal{Y}_r).\]
Now for $0 \le j \le m-1$,
\begin{equation*}
\begin{split}
\mathcal{Y}_r*\mathcal{V}_{l+j} &= \mathcal{V}_{l+j} + \sum_{i=r}^{n-1} \mathcal{V}_i*\mathcal{V}_{l+j} =\mathcal{V}_{l+j} + \sum_{i=r}^{n-1} (p-1)p^{n-i-1}\mathcal{V}_{l+j}\\
&=\mathcal{V}_{l+j}+ \mathcal{V}_{l+j}(p^{n-r}-1) = p^{n-r}\mathcal{V}_{l+j}.
 \end{split}
\end{equation*}
Hence
\[\mathcal{Y}_r*\mathcal{Y}_l = p^{n-r}(\mathcal{V}_{l} + \mathcal{V}_{l+1} +\cdots + \mathcal{V}_{l+m-1} +\mathcal{Y}_r) = p^{n-r}\mathcal{Y}_l.\]
\end{proof}
In the next proposition, we obtain relations for $\mathcal{U}_0$.
\begin{prop} \label{prop:reln2}
\begin{enumerate}
\item $\mathcal{U}_0*\mathcal{U}_0 = p^{n-1}(p-1)\mathcal{U}_0 + p^n\mathcal{Y}_1$.
\item $\mathcal{U}_0*\mathcal{Y}_r = p^{n-r}\mathcal{U}_0 = \mathcal{Y}_r*\mathcal{U}_0$ for all $1\le r\le n$.
\item $\mathcal{U}_0*(\mathcal{U}_0-p^n)*(\mathcal{U}_0+p^{n-1})=0$.
\end{enumerate}
\end{prop}
\begin{proof}
Note that
\[K_0w(1)K_0 = \bigsqcup_{s \in \Z_p/p^nZ_p} \alpha_sK_0 \quad \text{where $\alpha_s =x(s)w(1)$}.\]
To compute $\mathcal{U}_0*\mathcal{U}_0$ need to check if it supported on $1$, $w(1)$ and
$y(p^{j})$ for $1\le j \le n-1$, i.e. need to check if there exists $s,\ t$ such that
\[(\alpha_s\alpha_t)^{-1} g = \mat{-1}{s}{-t}{st-1}g \in K_0.\]
For $g=1$ taking $s=t=0$, for $g=w(1)$ taking $s=t=1$ and for $g=y(p^{j})$, taking $s=p^{n-j}$, $t=-p^{j}$ we get that
$\mathcal{U}_0*\mathcal{U}_0$ is supported on $K_0$, $K_0w(1)K_0$ and $K_0y(p^j)K_0$ for all $1 \le j \le n-1$.
Clearly $\mathcal{U}_0*\mathcal{U}_0(1) = p^n$. Doing similar calculations as before we get that
\[\mathcal{U}_0*\mathcal{U}_0(w(1))=\# \{s \in \Z_p/p^nZ_p : s \not\in p\Z_p\} = p^{n-1}(p-1),\]
and
\[\mathcal{U}_0*\mathcal{U}_0(y(p^{j})) = p^n \quad \text{ for $1 \le j \le n-1$}. \]
Thus
\[\mathcal{U}_0*\mathcal{U}_0 = p^{n-1}(p-1) \mathcal{U}_0 + p^n(I +\mathcal{V}_1 + \cdots \mathcal{V}_{n-1} ) = p^{n-1}(p-1) \mathcal{U}_0 + p^n\mathcal{Y}_1.\]

Similarly we can check that for each $1\le j\le n-1$,
$\mathcal{U}_0*\mathcal{V}_j$ and $\mathcal{V}_j*\mathcal{U}_0$ are supported only on $K_0w(1)K_0$ and that
\[\mathcal{U}_0*\mathcal{V}_j = \mathcal{V}_j*\mathcal{U}_0 =(p-1)p^{n-j-1}\mathcal{U}_0\]
which implies $(2)$.

The statement $(3)$ now follows using $(1)$ and $(2)$.
\end{proof}
Thus we have the following theorem.
\begin{thm}
The algebra $H(K//K_0(p^n))$ is an $n+1$ dimensional commutative algebra 
with generators $\{\mathcal{U}_0,\ \mathcal{Y}_1,\ \mathcal{Y}_2,\ \ldots,\ \mathcal{Y}_{n} \}$ and
relations given by Corollary~\ref{cor:reln1} and Proposition~\ref{prop:reln2}.
\end{thm}

We should point out that we have not yet found an analogue of Theorem~\ref{thm:thm1Hecke} for $H(G//K_0(p^n))$ for $n \ge 2$. However
we would need the following relation later. Let $\mathcal{T}_m = X_{d(p^m)}$, $\mathcal{U}_m = X_{w(p^m)}$, $\mathcal{Z}= X_{z(p)}$ be the elements in $H(G//K_0(p^n))$. Then
\begin{lem}\label{lem:rel5}
$(\mathcal{T}_1)^m *\mathcal{U}_m = \mathcal{T}_m *\mathcal{U}_m = \mathcal{Z}^m * \mathcal{U}_0$ for all $m \le n$.
\end{lem}
\begin{proof}
The proof follows as before by using Lemma~\ref{lem:rel1} and since
\[
K_0(p^n) d(p^m) K_0(p^n) = \bigsqcup_{s\in \Z_p / p^{m}\Z_p}x(s)d(p^m)K_0(p^n) \quad \text{for $m \ge 0$},
\]
and
\[
K_0(p^n) w(p^r) K_0(p^n) = \bigsqcup_{s\in \Z_p / p^{n-r}\Z_p}x(s)w(p^r)K_0(p^n) \quad \text{for $r \le n$}.
\]
\end{proof}

\subsection{Representations of $K$ having a $K_0(p^n)$ fixed vector} 
In this section we recall some results of 
Casselman~\cite{Casselman}~\cite{Casselman2}.
We are interested in irreducible representations of
$K$ having a $K_0(p^n)$ fixed vector. Let
\[I(n) := Ind^K_{K_0(p^n)} 1 = \{ \phi:K\rightarrow \C : \phi(k_0k)=\phi(k) \text{ for } k_0\in K_0(p^n),\ k \in K\}.\]
Then $I(n)$ is a right representation of $K$, via right translation, denoted by $\pi_R$, where
$\pi_R(k)(\phi)(k')=\phi(k'k)$, and the dimension of this representation is $[K:K_0(p^n)] = p^{n-1}(p+1)$.
It follows from Frobenius Reciprocity that every (smooth) irreducible
representation of $K$ which has a nonzero $K_0(p^n)$ fixed vector
is isomorphic to a subrepresentation of $I(n)$.
We shall therefore decompose $I(n)$ into sum of
irreducible representations.

The following lemma is clear.
\begin{lem}\label{lem:represen1}
We have $I(n)^{K_0(p^n)} =  H(K//K_0(p^n))$ and consequently the dimension of $I(n)^{K_0(p^n)}$ is $n+1$.
\end{lem}

Using induction argument and Frobenius reciprocity we obtain following 
well-known results. 
\begin{prop}\label{prop:represen1}
The representation $I(n)$ is a sum of $n+1$
distinct irreducible representations.
\end{prop}
\begin{cor}\label{cor:represen1}
Let $n\geq 0$. There exists a unique irreducible representation $\sigma(n)$ of $K$ such that $\sigma(n)$
has a $K_0(p^n)$ fixed vector and such that $\sigma(n)$ does not have a $K_0(p^k)$ fixed vector
for $k<n$. Further, $\sigma(n)$ has a unique $K_0(p^n)$ fixed vector up to scalar multiplication and
the dimension of $\sigma(n)$ is given by: $\text{dim}(\sigma(0))=1$, $\text{dim}(\sigma(1))=p$ and
$\text{dim}(\sigma(n))=p^{n-2}(p^2-1)$ for $n\geq 2$.
\end{cor}
We have the following theorem of Casselman. 
\begin{thm} (Casselman~\cite{Casselman})
Let $(\pi,V)$ be an irreducible admissible representation of
$G=\GL_2(\mathbb{Q}_p)$ with trivial central character. Let $n$ be
the minimal integer such that there exists a nonzero $K_0(p^n)$ 
fixed vector in $V$. Then this vector is unique up to a scalar.
\end{thm}
We shall now explicitly describe the irreducible subrepresentations of $I(n)$.
Let us consider the action $\pi_L$ of $H(K//K_0(p^n))$ on $I(n)$:\\
for $f \in H(K//K_0(p^n))$ and $\phi \in I(n)$ set
\[\pi_L(f)(\phi)(g) = \int_K f(k)\phi(k^{-1}g) dk \quad \text{ for all $g \in K$}.\]
In particular, if $\phi \in I(n)^{K_0(p^n)}$ which by Lemma~\ref{lem:represen1} is same as the algebra $H(K//K_0(p^n))$
then we have $\pi_L(f)(\phi) = f * \phi$. It is easy to check that the action $\pi_L$
commutes with the action $\pi_R$. It now follows by Schur's Lemma that for each $f \in H(K//K_0(p^n))$ the operator
$\pi_L(f)$ acts as a scalar operator on an irreducible subrepresentation of $I(n)$. We shall use this to distinguish the
irreducible components of $I(n)$ as follows.

If $\sigma$ is any irreducible subrepresentation of $I(n)$ then $\sigma$ contains a
$K_0(p^n)$ fixed vector, that is there exists a non-zero vector $v_{\sigma} \in \sigma \cap I(n)^{K_0(p^n)}$. Thus $v_\sigma$ is a linear
combination of $\mathcal{U}_0$ and $\mathcal{Y}_r$ for $1 \le r \le n$. 
Since $\pi_L(f)$ acts as a scalar for every $f \in H(K//K_0(p^n))$ the vector
$v_\sigma$ will be an eigenvector under the action of $\pi_L(\mathcal{U}_0)$ and $\pi_L(\mathcal{Y}_r)$ for all $1 \le r \le n$. For each $\sigma$ we can
compute these eigenvectors $v_\sigma$ and their corresponding eigenvalues using the relations in
Corollary~\ref{cor:reln1} and Proposition~\ref{prop:reln2}.
In fact we obtain the following proposition.
\begin{prop}\label{prop:represen2}
A basis of eigenvectors for $H(K//K_0(p^n))$ under the above action is given by:\\
$v_1 = \mathcal{U}_0+\mathcal{Y}_1$\\
$v_2 = \mathcal{U}_0 - p\mathcal{Y}_1$\\
$w_k = \mathcal{Y}_{k} - p\mathcal{Y}_{k+1}$ for $1 \le k \le n-1$,\\
with eigenvalues given by the following table:
\begin{table}[ht]
\begin{tabular}{c|ccccccccc}
    & $\mathcal{U}_0$ & $\mathcal{Y}_1$ & $\mathcal{Y}_2$ & $\mathcal{Y}_3$ & $\hdots$ & $\mathcal{Y}_{k}$ & $\hdots$ & $\mathcal{Y}_{n-1}$ & $\mathcal{Y}_{n}$  \\
[0.5ex]
\hline
$v_1$ & $p^n$ & $p^{n-1}$ & $p^{n-2}$ & $p^{n-3}$ & $\hdots$ & $p^{n-k}$ & $\hdots$ & $p$ & $1$ \\
$v_2$ & $-p^{n-1}$ & $p^{n-1}$ & $p^{n-2}$ & $p^{n-3}$ & $\hdots$ & $p^{n-k}$ & $\hdots$ & $p$ & $1$ \\
$w_1$ & $0$ & $0$ & $p^{n-2}$ & $p^{n-3}$ & $\hdots$ & $p^{n-k}$ & $\hdots$ & $p$ & $1$ \\
$w_2$ & $0$ & $0$ & $0$ & $p^{n-3}$ & $\hdots$ & $p^{n-k}$ & $\hdots$ & $p$ & $1$ \\
$\vdots$ & $\vdots$ & $\vdots$ & $\vdots$ & $\vdots$ & $\vdots$ & $\vdots$ & $\vdots$ & $\vdots$ & $\vdots$ \\
$w_k$ & $0$ & $0$ & $0$ & $0$ & $\hdots$ & $p^{n-k}$ & $\hdots$ & $p$ & $1$ \\
$\vdots$ & $\vdots$ & $\vdots$ & $\vdots$ & $\vdots$ & $\vdots$ & $\vdots$ & $\vdots$ & $\vdots$ & $\vdots$ \\
$w_{n-2}$ & $0$ & $0$ & $0$ & $0$ & $\hdots$ & $0$  & $\hdots$ & $p$ & $1$ \\
$w_{n-1}$ & $0$ & $0$ & $0$ & $0$ & $\hdots$ & $0$ & $\hdots$ & $0$ & $1$
\end{tabular}
\end{table}\\
where each entry of the table  at the intersection of the row $v$ and coloumn $F$ stands for the eigenvalue of the action of $F$ on $v$,
for example, $\mathcal{U}_0 * v_1 = p^n v_1$.
\end{prop}
\begin{cor}
The representation $I(n)$ is a sum of $n+1$ irreducible subspaces given by:
$S_1 = \Span(\pi_R(K)v_1)$, $S_2 = \Span(\pi_R(K)v_2)$ and $T_k = \Span(\pi_R(K)w_k)$
where $1\le k \le n-1$ such that $\mathrm{dim}(S_1) = 1$, $\mathrm{dim}(S_2) = p$,
$\mathrm{dim}(T_k) = p^{k-1}(p^2-1)$. By Corollary~\ref{cor:represen1},  
$T_{n-1} = \sigma_n$ and hence is the unique irreducible
representation of $K$ such that $T_{n-1}$ has a $K_0(p^n)$ fixed vector $w_{n-1}$
but does not have $K_0(p^k)$ fixed vector for $k<n$.
\end{cor}
\begin{proof}
It follows from the above table that the set of eigenvalues for vectors 
$v_i$ for $i=1,\ 2$ and $w_k$ for $1 \le k \le n-1$ are distinct and hence 
each of them lies in an irreducible component. To finish the proof we 
need to compute the dimensions, for which we shall need
the following lemma. A statement similar to this lemma appears in 
\cite{L-S}.

\begin{lem}\label{lem:represen2}
The operators $\pi_L(\mathcal{U}_0)$ and $\pi_L(\mathcal{V}_r)$ for $1 \le r \le n-1$ have trace zero.
\end{lem}
\begin{proof}
For $g \in K$, let $\phi_g$ be the characteristic function of $K_0(p^n)g$, then
$I(n)$ as a complex vector space has a basis consisting of $\phi_g$ as $g$
varies over the right coset representatives of $K$ modulo $K_0(p^n)$. Thus to prove
lemma it is enough to show that $\pi_L(\mathcal{U}_0)(\phi_g)(g) =\pi_L(\mathcal{V}_r)(\phi_g)(g)=0$, we
will show it for $\mathcal{V}_r$, for $\mathcal{U}_0$ the same argument works. It is easy to see that $\pi_L(\mathcal{V}_r)(\phi_g)$ is supported on $K_0(p^n)y(p^r)K_0(p^n)g$.
So if $\pi_L(\mathcal{V}_r)(\phi_g)(g) \ne 0$ then $g \in K_0(p^n)y(p^r)K_0(p^n)g$ which is
impossible as $K_0(p^n) \ne K_0(p^n)y(p^r)K_0(p^n)$.
\end{proof}
Using table in Proposition~\ref{prop:represen2}, it is easy to obtain following table where
we consider the action of $\mathcal{U}_0,\ \mathcal{V}_1,\ \mathcal{V}_2,\ \ldots,\ \mathcal{V}_{n-1}$ instead:
\begin{table}[ht]
\begin{tabular}{c|ccccccc}
      & $\mathcal{U}_0$ & $\mathcal{V}_1$  & $\hdots$ & $\mathcal{V}_k$ & $\hdots$ & $\mathcal{V}_{n-2}$ & $\mathcal{V}_{n-1}$  \\
[0.5ex]
\hline
$v_1$ & $p^n$ & $p^{n-2}(p-1)$ & $\hdots$ & $p^{n-k-1}(p-1)$ & $\hdots$ & $p(p-1)$ & $p-1$ \\
$v_2$ & $-p^{n-1}$ & $p^{n-2}(p-1)$ & $\hdots$ & $p^{n-k-1}(p-1)$ & $\hdots$ & $p(p-1)$ & $(p-1)$ \\
$w_1$ & $0$ & $-p^{n-2}$  & $\hdots$ & $p^{n-k-1}(p-1)$ & $\hdots$ & $p(p-1)$ & $(p-1)$ \\
$w_2$ & $0$ & $0$  & $\hdots$ & $p^{n-k-1}(p-1)$ & $\hdots$ & $p(p-1)$ & $(p-1)$ \\
$\vdots$ & $\vdots$  & $\vdots$ & $\vdots$ & $\vdots$ & $\vdots$ & $\vdots$ & $\vdots$ \\
$w_k$ & $0$ & $0$ & $\hdots$ & $-p^{n-k-1}$ & $\hdots$ & $p(p-1)$ & $(p-1)$ \\
$\vdots$ &$\vdots$  & $\vdots$ & $\vdots$ & $\vdots$ & $\vdots$ & $\vdots$ & $\vdots$ \\
$w_{n-2}$ & $0$ & $0$  & $\hdots$ & $0$  & $\hdots$ & $-p$ & $(p-1)$ \\
$w_{n-1}$ & $0$ & $0$ & $\hdots$ & $0$ & $\hdots$ & $0$ & $-1$
\end{tabular}
\end{table}

Let $d_1$, $d_2,\ \ldots,\ d_{n+1}$ be the dimension of $S_1,\ S_2,\ \ldots,\ T_{n-1}$ respectively.
Then using Lemma~\ref{lem:represen2} and above table we have following system of linear equations:
\begin{eqnarray*}
p^nd_1-p^{n-1}d_2 &=&0 \\
p^{n-2}(p-1)d_1 + p^{n-2}(p-1)d_2 - p^{n-2}d_3 &=&0\\
\vdots \\
p^{n-k-1}(p-1)(d_1 + d_2 + d_3 + \cdots + d_{k+1}) - p^{n-k-1}d_{k+2} &=&0\\
\vdots \\
(p-1)(d_1 + d_2 + d_3 + \cdots + d_n) - d_{n+1} &=&0\\
d_1 + d_2 + \cdots + d_{n-1} &=& p^{n-1}(p+1)
\end{eqnarray*}
solving which we get the dimensions.
\end{proof}

\section{ Translation from the adelic setting to the classical setting}

In this section following Gelbart~\cite{Gelbart} we shall review
the connection between automorphic forms and classical modular forms and use this connection
to translate the adelic operators of the previous section into their classical
counterparts and thereby obtaining relations satisfied by them.

Let $\Hup$ be the upper half plane and
$G_\infty=\GL_2(\R)^{+}$. Then $G_{\infty}$ acts on $\Hup$
in a standard way.
For $g=\mat{a}{b}{c}{d} \in G_{\infty}$ and $z\in \Hup$ define
\[
j(g,z)=det(g)^{-1/2}(cz+d),\]
for $f$ functions on $\Hup$ define the slash operator $|_{2k}g$ by
\[f|_{2k}g = j(g,z)^{-2k}f\left(\frac{az+b}{cz+d}\right).\]

Let $\A=\A_\Q$ be the adele ring of $\Q$ and $Z_{\A}$ denotes the center of
$\GL_2(\A)$. Let $N$ be a positive integer. We let
$K_l=\GL_2(\Z_l)$ for a prime $l$ not dividing $N$ and
let $K_p=K_{0}(p^{\alpha})$ for a prime $p$ such that $p^{\alpha}\|N$.
Let $K_f$ be the subgroup of $\GL_2(\A)$ defined by
\[
K_f(N)=\prod_{q < \infty}K_q.
\]
By the strong approximation theorem we have
\begin{equation*}
\GL_2(\A)=\GL_2(\Q)G_{\infty}K_f(N)
\end{equation*}
We denote by $A_{2k}(N)$ the space of functions $\Phi:\GL_2(\A)\rightarrow \C$ satisfying the following properties:
\begin{enumerate}
 \item $\Phi(\gamma z g k)= \Phi(g)$ for all $\gamma \in \GL_2(\Q),\ z \in Z_\A,\ g \in \GL_2(\A),\ k \in K_f(N)$.
 \item $\Phi(g r(\theta)) = e^{-i2k\theta}\Phi(g)$ where $r(\theta) = \mat{cos\theta}{-sin\theta}{sin\theta}{cos\theta} \in \mathrm{SO}(2)$.
 \item $\Phi$ is smooth as a function of $G_\infty$ and satisfies the differential equation $\Delta \Phi = -k(k-1)\Phi$ where $\Delta$
 is the Casimir operator.
 \item $\Phi \in \mathrm{L}^2(Z_\A\GL_2(\Q)\backslash\GL_2(\A))$. 
 \item $\Phi$ is cuspidal, that is $\int_{\Q \backslash \A} \Phi\left(\mat{1}{a}{0}{1}g\right) da =0$ for all $g \in \GL_2(\A)$.
\end{enumerate}
By Gelbart \cite[Proposition 3.1]{Gelbart} there exists an isomorphism
\[A_{2k}(N)\rightarrow S_{2k}(\Gamma_0(N))\] given by $\Phi \mapsto f_{\Phi}$
where for $z \in \Hup$,
\[
f_{\Phi}(z)=\Phi(g_\infty)j(g_\infty,i)^{2k}
\]
where $g_\infty \in G_\infty$ is such that $g_\infty(i)=z$. The inverse map is given by $f\mapsto \Phi_f$ where for $g \in \GL_2(\A)$ if
$g=\gamma g_{\infty}k$ (using strong approximation),
\[
\Phi_f(g)=f(g_{\infty}(i))j(g_{\infty},i)^{-2k}.
\]
This isomorphism induces a ring isomorphism of spaces of linear operators by
\[q:\mathrm{End}_{\C}(A_{2k}(N))\rightarrow \mathrm{End}_{\C}(S_{2k}(\Gamma_0(N)))\] given by
\[
q(\mathcal{T})(f)=f_{\mathcal{T}(\Phi_f)}.\]

Let $N=p^nM$ where $p$ is a prime coprime to $M$ and $G=\GL_2(\Q_p)$.
We note that the $H(G//K_0(p^n))$ is a subalgebra of $\mathrm{End}_{\C}(A_{2k}(N))$ via the following
action:
\[\text{for $\mathcal{T} \in H(G//K_0(p^n))$ and $\Phi \in A_{2k}(N)$,}\ \ \mathcal{T}(\Phi)(g) = \int_{G}\mathcal{T}(x) \Phi(gx)dx.\]
\begin{remark}~\label{rem:1}
We note that if $p_1$ and $p_2$ are distinct primes then 
the operators $\mathcal{T}_1 \in H(G//K_0(p_1^n))$ and 
$\mathcal{T}_2 \in H(G//K_0(p_2^n))$ in $\mathrm{End}_{\C}(A_{2k}(N))$ commute, that is, 
$\mathcal{T}_1\circ\mathcal{T}_2 = \mathcal{T}_2\circ\mathcal{T}_1$.
\end{remark}

Then we have following propositions.
\begin{prop}\label{prop:trans1}
Let $N=pM$ such that $p \nmid M$ and $f \in S_{2k}(\Gamma_0(N))$. For $\mathcal{T}_1,\ \mathcal{U}_1 \in H(G//K_0(p))$ we have
\begin{enumerate}
 \item $q(\mathcal{T}_1)(f)(z) = p^{-k} \sum_{s=0}^{p-1} f((z+s)/p)$.
 \item $q(\mathcal{U}_1)(f)(z) = f|_{2k}\mat{p\beta}{1}{N\gamma}{p} (z) \quad \text{where $p^{2}\beta-N\gamma=p$}$.
\end{enumerate}
\end{prop}
\begin{proof}
For $\Phi \in A_{2k}(N)$ we have using Lemma~\ref{lem:rel3},
\[\mathcal{T}_1(\Phi)(g) = \int_{G}X_{d(p)}(x) \Phi(gx)dx = \int_{K_0d(p)K_0} \Phi(gx)dx,\]
\[=\sum_{s=0}^{p-1} \int_{x(-s)d(p)K_0} \Phi(gx) = \sum_{s=0}^{p-1} \Phi(gx(-s)d(p)),\]
and
\[\mathcal{U}_1(\Phi)(g) = \int_{G}X_{w(p)}(x) \Phi(gx)dx = \int_{w(p)K_0} \Phi(gx)dx = \Phi(gw(p)).\]
Hence for $(1)$ we get
\[q(\mathcal{T}_1)(f)(z) = f_{\mathcal{T}_1(\Phi_f)}(z) = \sum_{s=0}^{p-1} \Phi_f(g_\infty x(-s)d(p))j(g_\infty, i)^{2k},\]
where $g_\infty \in G_\infty$ such that $g_\infty i = z$. Since $\Phi_f$ is invariant under left multiplication by rational
matrices, multiplying by $\gamma = d(p^{-1})x(s) \in \GL_2(\Q)$ we obtain
\[\Phi_f(g_\infty x(-s)d(p)) = \Phi_f(d(p^{-1})x(s)g_\infty \cdot k_f) = \Phi_f(d(p^{-1})x(s)g_\infty) \]
where $k_f \in K_f(N)$ has $1$ in the $p$-th place and $d(p^{-1})x(s)$ in $q$-th place for any finite prime $q \ne p$.
Thus,
\begin{equation*}
\begin{split}
q(\mathcal{T}_1)(f)(z) &= \sum_{s=0}^{p-1} \Phi_f(d(p^{-1})x(s)g_\infty)j(g_\infty, i)^{2k}\\
= \sum_{s=0}^{p-1} f(d(p^{-1})x(s)z) & j(d(p^{-1})x(s), z)^{-2k} = p^{-k} \sum_{s=0}^{p-1} f((z+s)/p).
\end{split}
\end{equation*}
For $(2)$, let $W_p = \mat{p\beta}{1}{N\gamma}{p}$ be a matrix of determinant $p$.
As before, multiplying $g_\infty w(p)$ by $W_p z(p^{-1}) \in \GL_2(\Q)$ we get
\begin{equation*}
q(\mathcal{U}_1)(f)(z) = \Phi_f(g_\infty w(p))j(g_\infty, i)^{2k} =\Phi_f(W_pz(p^{-1})g_\infty \cdot k_f)j(g_\infty, i)^{2k}
\end{equation*}
where $k_f$ has $p$-th component $W_pz(p^{-1})w(p) \in K_p$ and for prime $q \ne p$ has
$q$-th component $W_pz(p^{-1}) \in K_q$, that is $k_f \in K_f(N)$. Thus,
\[q(\mathcal{U}_1)(f)(z)= 
 f(W_pz)j(W_p, z)^{-2k} = f|_{2k}W_p(z).\]
\end{proof}
\begin{remark}
The operator $q(\mathcal{U}_1)$ is the usual Atkin-Lehner operator $W_p$ while the
operator $q(\mathcal{T}_1)$ is the operator $\tilde{U}_p=p^{1-k}U_p$ where $U_p$ is the usual Hecke operator, sometime also
denoted as $T_p$ (Refer to \cite{A-L} and \cite{Miyake} for more details). It is obvious that $q(\mathcal{Z})$ is the identity
operator.
\end{remark}
Let $Q_p = q(\mathcal{U}_0)$ where $\mathcal{U}_0 \in H(G//K_0(p))$. Then using Lemma~\ref{lem:rel4} we have
\begin{cor}
$Q_p= p^{1-k}U_pW_p$ and $(Q_p-p)(Q_p+1)=0$.
\end{cor}

\begin{prop}\label{prop:trans2}
Let $N=p^nM$ where $n \ge 2$ and $p \nmid M$. Let $f \in S_{2k}(\Gamma_0(N))$. For 
$\mathcal{T}_1,\ \mathcal{U}_m,\ \mathcal{V}_r \in H(G//K_0(p^n))$ where $1 \le r \le n-1$, $m \le n$ we have
\begin{enumerate}
 \item $q(\mathcal{T}_1)(f)(z) = p^{-k} \sum_{s=0}^{p-1} f((z+s)/p) = \tilde{U}_p(f)(z)$.
 \item If $f \in S_{2k}(\Gamma_0(p^rM))$ where $r \le n$ then 
 $q(\mathcal{U}_r)(f)(z) = p^{n-r}f|_{2k}W_{p^r} (z) $ where $W_{p^r} = \mat{p^r\beta}{1}{p^rM\gamma}{p^r}$ 
 is an integer matrix of determinant $p^r$. In particular, 
 $q(\mathcal{U}_n)(f)(z) = f|_{2k}W_{p^n} (z) $ where $W_{p^n}$ is the Atkin-Lehner operator on $S_{2k}(\Gamma_0(N))$. 
 \item $q(\mathcal{V}_r)(f)(z) = \sum_{s \in \Z_p^*/ 1+p^{n-r}\Z_p} f|_{2k}A_s$
 where $A_s \in \SL_2(\Z)$ is any matrix of the form 
 $\mat{a_s}{b_s}{p^{r}M}{p^{n-r}-sM}$.
 \item If $f \in S_{2k}(\Gamma_0(p^rM))$ then $q(\mathcal{V}_r)(f) = p^{n-r-1}(p-1)f$, 
 consequently, $q(\mathcal{Y}_r)(f) = p^{n-r}f$.
\end{enumerate}
\end{prop}
\begin{proof}
The proof of $(1)$ is as in Proposition~\ref{prop:trans1}. The proof of $(2)$ is similar, using decomposition in Lemma~\ref{lem:rel5} we have 
for $r \le n$,
\[\mathcal{U}_r(\Phi)(g) = \sum_{s=0}^{p^{n-r}-1} \Phi(gx(s)w(p^r)).\]
Let $f \in S_{2k}(\Gamma_0(p^rM))$ and $W_{p^r} = \mat{p^r\beta}{1}{p^rM\gamma}{p^r}$ be an integer matrix of determinant $p^r$. Then,
\[q(\mathcal{U}_r)(f)(z)= \sum_{s=0}^{p^{n-r}-1} \Phi_f(g_\infty x(s)w(p^r)) j(g_\infty, i)^{2k}\] where 
$z = g_\infty i$. Since $\Phi_f \in A_{2k}(p^rM)$ multiplying $g_\infty x(s)w(p^r)$ by the matrix $W_{p^r} z(p^{-r})x(-s) \in \GL_2(\Q)$ 
we get that, 
\[\Phi_f(g_\infty x(s)w(p^r)) = \Phi_f(h_\infty k_f) = \Phi_f(h_\infty) \] 
where $h_\infty = z(p^{-r})W_{p^r}x(-s)g_\infty \in G_\infty$ and $k_f \in K_f(p^rM)$. 
Since $f |_{2k} W_{p^r} \in S_{2k}(\Gamma_0(p^rM))$,
\[q(\mathcal{U}_r)(f)(z)= \sum_{s=0}^{p^{n-r}-1} f |_{2k} W_{p^r}x(-s) (z) = p^{n-r} f |_{2k} W_{p^r} (z).\]

For $(3)$, if $\Phi \in A_{2k}(N)$ then using Lemma~\ref{lem:lem1p^r} and~\ref{lem:lem2p^r} we have
\begin{equation*}
\mathcal{V}_r(\Phi)(g) = \int_{G}X_{y(p^r)} \Phi(gh) dh
=\sum_{s \in \Z_p^*/ 1+p^{n-r}\Z_p} \Phi(gd(s)y(p^r)).
\end{equation*}
Let $z \in \Hup$ be such that $z=g_{\infty} i$ for some $g_{\infty} \in \mathrm{G}_{\infty}$. Then,
\[q(\mathcal{V}_r)(f)(z)  
=\sum_{s \in \Z_p^*/ 1+p^{n-r}\Z_p}\Phi_f(g_{\infty}d(s)y(p^r))j(g_{\infty}, i)^{2k}. \]
By the strong approximation, $g_{\infty}d(s)y(p^r) = A_s^{-1} h_{\infty} k_f$ for some $A_s \in \GL_2(\Q)$, $h_\infty \in \mathrm{G}_{\infty}$
and $k_f \in K_f(N)$. So we need $A_s \in \GL_2(\Q)$ such that $A_sd(s)y(p^r)$ belongs to $K_0(p^n)$ and $A_s$ belongs to $K_q$ for $q \ne p$.
So we must choose $A_s$ with determinant $1$. For any $s \in \Z_p^*$, we have $\mathrm{gcd}(p^rM, p^{n-r}-sM)=1$, so there
exists integers $a_s$, $b_s$ such that $a_s(p^{n-r}-sM)-b_sp^rM=1$. Take
\[A_s=\mat{a_s}{b_s}{p^rM}{p^{n-r}-sM} \in \SL_2(\Z),\]
then $A_s$ belongs to $K_q$ for $q \ne p$ and
\[A_sd(s)y(p^r) 
= \mat{a_s+b_sp^r}{b_s}{p^n}{p^{n-r}-sM}\in K_0.\]
Thus
\[\Phi_f(g_{\infty}d(s)y(p^r)) = f(A_sz) j(A_s,z)^{-2k} j(g_{\infty},i)^{-2k},\]
and so
\[q(\mathcal{V}_r)(f)(z) = \sum_{s \in \Z_p^*/ 1+p^{n-r}\Z_p} f(A_sz) j(A_s,z)^{-2k} = \sum_{s \in \Z_p^*/ 1+p^{n-r}\Z_p} f|_{A_s}(z). \]
Thus if $f \in  S_{2k}(\Gamma_0(p^r M))$ then $q(\mathcal{V}_r)(f)(z) = p^{n-r-1}(p-1)f$.
Further,
\[q(\mathcal{Y}_r)(f) = f + \sum_{j=r}^{n-1} \sum_{s \in \Z_p^*/1+p^{n-j}\Z_p}f|_{2k}\mat{a_{s,j}}{b_{s,j}}{p^{j}M}{p^{n-j}-sM}\]
\[=  f + \sum_{j=r}^{n-1} (p-1) p^{n-j-1} f = p^{n-r}f,\]
proving $(4)$.
\end{proof}

Let $N=p^nM$ with $p$ and $M$ coprime and $n \ge 2$.
Let $Q_{p^m} = (\tilde{U}_p)^m W_{p^m}$ for $m \le n$ where $W_{p^m}$ is the Atkin-Lehner operator 
on $S_{2k}(\Gamma_0(p^mM))$. Using 
Lemma~\ref{lem:rel5} and Propositions~\ref{prop:trans2} and~\ref{prop:reln2} we have
\begin{cor}\label{cor:trans3}
For $\mathcal{U}_0 \in H(G//K_0(p^n))$, we have
$Q_{p^n} = q(\mathcal{U}_0)$ and hence $Q_{p^n}(Q_{p^n}-p^n)(Q_{p^n}+p^{n-1})=0$.
Further for $m \le n$ we have $Q_{p^n} = (\tilde{U}_p)^m q(\mathcal{U}_m)$, hence if
$ f \in S_{2k}(\Gamma_0(p^mM)) \subseteq S_{2k}(\Gamma_0(N))$ then $Q_{p^n}(f) = p^{n-m}Q_{p^m}(f)$.
\end{cor}

Let $S_{p^n, r} = q(\mathcal{Y}_r)$ where $\mathcal{Y}_r \in H(G//K_0(p^n))$, $1 \le r \le n$. 
Using relations in Corollary~\ref{cor:reln1}, we have
\begin{cor}\label{cor:trans4}
 $S_{p^n, r}(S_{p^n, r}-p^{n-r})=0$
for $1 \le r \le n$.
\end{cor}

\section{Eigenspaces of classical operators and the characterization of the new space.}

Let $N$ be a positive integer. In this section we shall look at the classical operators on 
$S_{2k}(\Gamma_0(N))$ that come from the adelic Hecke algebra via the isomorphism
\[q:\mathrm{End}_{\C}(A_{2k}(N))\rightarrow \mathrm{End}_{\C}(S_{2k}(\Gamma_0(N)))\]
and study their eigenspaces. We shall prove the theorems stated in 
Section~\ref{sec:mainresults} including our main result Theorem~\ref{thm:sec2thm4}.
\subsection{$N$ square-free}\label{section:pM}
Let $N$ be a square-free positive integer and $S$ be the set of prime divisors of $N$.
Let $p \in S$. Recall that for
$\mathcal{U}_0,\ \mathcal{U}_1,$ and $\mathcal{T}_1 \in H(G//K_0(p))$ we respectively obtained the classical operators
$Q_p,\ W_p$ and $\tilde{U}_p$ where
\[\tilde{U}_p(f)(z) = p^{-k}\sum_{s=0}^{p-1}f((z+s)/p),\]
\[W_p(f)(z) = f|_{2k}\mat{p\beta}{1}{N\gamma}{p}(z) \quad \text{where $p^{2}\beta-N\gamma=p$},\]
\[Q_p(f)(z) = \tilde{U}_p W_p(f)(z) = p^{-k}\sum_{s=0}^{p-1}W_p(f)((z+s)/p),\]
and \[(Q_p-p)(Q_p+1)=0.\]
For $N$, $d$ any positive integers recall the shift operator 
$V(d) : S_{2k}(\Gamma_0(N)) \rightarrow S_{2k}(\Gamma_0(dN))$ given by 
$V(d)(f) = d^{-k}f|_{2k}\mat{d}{0}{0}{1}$.
It is well known \cite{A-L} that the old space
\begin{equation}\label{eq:1}
\begin{split}
S^\mathrm{old}_{2k}(\Gamma_0(N)) = & \bigoplus_{dM \mid N,\ M \ne N} V(d)S^\mathrm{new}_{2k}(\Gamma_0(M)) \\
= & 
\sum_{p_i \in S} S_{2k}(\Gamma_0(N/p_i)) + V(p_i)S_{2k}(\Gamma_0(N/p_i)). 
\end{split}
\end{equation}

We will consider the action of $Q_p$ on each of the above summands.
\begin{lem}\label{lem:lem1pM}
Let $f \in S^\mathrm{new}_{2k}(\Gamma_0(N))$ be a new form. Then $Q_p(f)=-f$, that is, 
$S^\mathrm{new}_{2k}(\Gamma_0(N))$ is contained in the $-1$ eigenspace of $Q_p$.
\end{lem}
\begin{proof}
By \cite[lemma 18]{A-L}, $S^\mathrm{new}_{2k}(\Gamma_0(N))$ has a basis of primitive forms,
so we can assume that $f$ is primitive. By \cite[Theorem 3]{A-L},
$W_p(f)=\lambda(p)f$ for some $\lambda(p)=\pm1$ and $U_p(f)=-\lambda(p)p^{k-1}f$.
Since $Q_p=p^{1-k}U_pW_p$ the result follows.
\end{proof}

Write $N=pM$ where $M$ is a square-free integer coprime to $p$.
\begin{lem}\label{lem:lem2pM}
Let f be a form in $S_{2k}(\Gamma_0(M)) \subset S_{2k}(\Gamma_0(N))$. 
Then $Q_p(f)=pf$.
\end{lem}
\begin{proof}
Since $\mat{\beta}{1}{M\gamma}{p}\in \Gamma_0(M)$ we have
\begin{equation*}
\begin{split}
W_p(f)((z+s)/p) &=p^k(N\gamma(z+s)/p+p)^{-2k}
f\left( \frac{p\beta(z+s)/p+1}{N\gamma(z+s)/p +p}\right)\\
&=p^k(M\gamma(z+s)+p)^{-2k}
f\left( \frac{\beta(z+s)+1}{M\gamma(z+s)+p}\right) \\
&=p^k f|_{2k}\mat{\beta}{1}{M\gamma}{p}(z+s)=p^k f(z+s) =p^kf(z)
\end{split}
\end{equation*}
Hence
\[Q_p(f)= p^{-k}\sum_{s=0}^{p-1}W_p(f)((z+s)/p) = pf(z).\]
\end{proof}
Next we consider action of $Q_p$ on the old subspace $V(p)(S_{2k}(\Gamma_0(M)))$.
\begin{lem}\label{lem:lem3pM}
Let $f\in S_{2k}(\Gamma_0(M))$ and $g(z)=f(pz) \in V(p)(S_{2k}(\Gamma_0(M)))$. Then
\[
Q_p(g)=p^{1-2k}T_p(f)-g,\]
where $T_p$ is the usual Hecke operator on $S_{2k}(\Gamma_0(M))$.
\end{lem}
\begin{proof}
Note that from \cite[Lemma 14]{A-L},
\[p^{1-2k}T_p(f) = f(pz) + p^{-2k}\sum_{s=0}^{p-1}f((z+s)/p).\]
As before we have
\begin{equation*}
\begin{split}
W_p(g)((z+s)/p) &=p^k(N\gamma(z+s)/p+p)^{-2k}
g\left( \frac{p\beta(z+s)/p+1}{N\gamma(z+s)/p +p}\right)\\
&=p^k(M\gamma(z+s)+p)^{-2k}
f\left( \frac{p\beta(z+s)+p}{M\gamma(z+s)+p}\right) \\
&=p^{-k}(M\gamma(z+s)/p+1)^{-2k}
f\left( \frac{p\beta(z+s)/p+1}{M\gamma(z+s)/p+1}\right)\\
&=p^{-k}f|_{2k}\mat{p\beta}{1}{M\gamma}{1}((z+s)/p)
=p^{-k}f((z+s)/p)
\end{split}
\end{equation*}
since $\mat{p\beta}{1}{M\gamma}{1}\in \Gamma_0(M)$.
Thus
\[Q_p(g)(z)
=p^{-2k}\sum_{s=0}^{p-1} f((z+s)/p)=p^{1-2k}T_p(f)-g.\]
\end{proof}

Let $X_p:=S_{2k}(\Gamma_0(M)) \oplus V(p)S_{2k}(\Gamma_0(M))$ be the subspace 
of $S_{2k}(\Gamma_0(N))$.
\begin{cor} \label{cor:eigspcpM}
$Q_p$ stabilizes $X_p$ and the $-1$ eigenspace of $Q_p$ inside $X_p$ consists of forms
$h(z)=-\frac{p^{1-2k}}{p+1}T_{p}(f)(z)+f(pz)$ where $f\in S_{2k}(\Gamma_0(M))$.
\end{cor}
\begin{proof}
Let $h \in X_p$ be an old form. Then
$h$ can be uniquely written as $h(z)=f_1(z)+g(z)$ where $g(z)=f(pz)$ for some $f,\ f_1\in S_{2k}(\Gamma_0(M))$.
By Lemma~\ref{lem:lem2pM} and Lemma~\ref{lem:lem3pM} we have $Q_p(h)=pf_1 + p^{1-2k}T_p(f)-g$ which is clearly in $X_p$. 

Further since the above decomposition
for $Q_p(h)$ is unique, if $f_1(z)+g(z)$ is an eigenfunction of $Q_p$ with $g \ne 0$ then $Q_p(h)=-h$ and 
$f_1= -\frac{p^{1-2k}}{p-1}T_p(f)$.
Hence $-1$ eigenspace of $Q_p$ consists of forms
$h(z)=-\frac{p^{1-2k}}{p+1}T_{p}(f)(z)+f(pz)$ for some $f\in S_{2k}(\Gamma_0(M))$.
\end{proof}

From Lemma~\ref{lem:lem2pM} and Corollary~\ref{cor:eigspcpM} to obtain following proposition.
\begin{prop}\label{prop:prop1pM}
The $p$ eigenspace of $Q_p$ in $X_p$ is $S_{2k}(\Gamma_0(M))$.
\end{prop}

Next consider the operator $Q_p'=W_{p} \tilde{U}_p$. So, 
$Q_p' = W_pQ_pW_p^{-1} =W_pQ_pW_p$ and $Q_p'$ satisfies the equation $(Q_p'-p)(Q_p'-1)=0$.
Note that $f$ is an eigenfunction of $Q_p$ with eigenvalue $\lambda$ if and only if 
$W_p(f)$ is an eigenfunction of $Q_p'$ with eigenvalue $\lambda$.
Since the action of Atkin-Lehner operator $W_p$ on the space of new forms is surjective,
$Q_p'$ acts with the eigenvalue $-1$ on the space of new forms. We have the following lemma.
\begin{lem}\label{lem:lem4pm}
Let $f \in S_{2k}(\Gamma_0(M))$. Then $W_p(f)(z) = p^kf(pz)$. 
Further if $g=f(pz)$, then
$W_p(g)(z) = p^{-k}f(z)$. Consequently $W_p$ maps $S_{2k}(\Gamma_0(M))$ onto 
$V(p)S_{2k}(\Gamma_0(M))$, so $V(p)S_{2k}(\Gamma_0(M))$ is contained in the
$p$ eigenspace of $Q_p'$. Further $Q_p'$ preserves $X_p$ and the 
$p$ eigenspace of $Q_p'$ in $X_p$ is the space $V(p)(S_{2k}(\Gamma_0(M))$.
\end{lem}
\begin{proof}
Since $\mat{\beta}{1}{M\gamma}{p}\in \Gamma_0(M)$ we get
\[W_p(f)(z) = f|_{2k}\mat{p\beta}{1}{N\gamma}{p}(z) =p^k(M\gamma(pz) +p)^{-2k} f\left( \frac{\beta(pz)+1}{M\gamma(pz) +p}\right)\]
\[= p^k f|_{2k}\mat{\beta}{1}{M\gamma}{p}(pz) = p^k f(pz).\]
Further, since $\mat{p\beta}{1}{M\gamma}{1}\in \Gamma_0(M)$ we get
\[W_p(g)(z) = g|_{2k}\mat{p\beta}{1}{N\gamma}{p}(z) =p^k(N\gamma z +p)^{-2k} f\left( \frac{p^2\beta z+p}{N\gamma z +p}\right)\]
\[= p^{-k}(M\gamma z +1)^{-2k} f\left( \frac{p\beta z+1}{M\gamma z +1}\right) = p^{-k} f(z).\]
Hence $W_p(X_p) = X_p$ and so $Q_p'$ preserves $X_p$. It now follows from Proposition~\ref{prop:prop1pM} that 
$p$ eigenspace of $Q_p'$ in $X_p$ is precisely the space $V(p)(S_{2k}(\Gamma_0(M))$.
\end{proof}

We shall need the following proposition.
\begin{prop}\label{prop:prop2pM}
The operators $Q_p=\tilde{U}_pW_p$ and $Q_p'= W_pQ_pW_p^{-1}$ are self-adjoint with respect to Petersson inner product.
\end{prop}
\begin{proof}
Recall that $\tilde{U}_p= p^{1-k}U_p = p^{1-k}T_p$ where $T_p$ is the usual Hecke operator on $S_{2k}(\Gamma_0(N))$.
Following Miyake~\cite[Page 135]{Miyake}
\[T_p = \Gamma_0(N) \mat{1}{0}{0}{p} \Gamma_0(N)\] and
for $f \in S_{2k}(\Gamma_0(N))$
 \[T_p(f)= f|_{2k} \Gamma_0(N) \mat{1}{0}{0}{p} \Gamma_0(N) =
p^{k-1} \sum_{m=0}^{p-1} f|_{2k} {\mat{1}{m}{0}{p}}.\]
Further,
\[T_p^* = \Gamma_0(N) \mat{p}{0}{0}{1} \Gamma_0(N)\] and
\[T_p^*(f)= f|_{2k} \Gamma_0(N) \mat{p}{0}{0}{1} \Gamma_0(N) =
p^{k-1} \sum_{m=0}^{p-1} f|_{2k} {\mat{p}{-m}{0}{1}}.\]
Thus for $f,\ g \in S_{2k}(\Gamma_0(N))$, by \cite[Theorem 2.8.2]{Miyake}, $\langle T_p(f),\ g \rangle = \langle f,\ T_p^*(g) \rangle$.

The Atkin-Lehner operator $W_p$ acts by a matrix $\mat{p\beta}{1}{N\gamma}{p}$ such that $p^2 \beta - N\gamma = p$. We
want to show that the following diagram commutes:
\[
\begin{tikzcd}
S_{2k}(\Gamma_0(N)) \arrow{r}{T_p} \arrow[swap]{d}{W_p} & S_{2k}(\Gamma_0(N)) \arrow{d}{W_p} \\
S_{2k}(\Gamma_0(N)) \arrow{r}{T_p^*} & S_{2k}(\Gamma_0(N))
\end{tikzcd}
\]
We have
\begin{equation}\label{eq:eq1}
\begin{split}
f|_{2k}W_{p}^{-1}T_pW_p &= f|_{2k} W_{p}^{-1}\Gamma_0(N) \mat{1}{0}{0}{p} \Gamma_0(N)W_p \\
&=f|_{2k} \Gamma_0(N)W_{p}^{-1} \mat{1}{0}{0}{p} W_p\Gamma_0(N)
\end{split}
\end{equation}
since $W_p\Gamma_0(N)W_{p}^{-1} = \Gamma_0(N)$.

We claim that $\Gamma_0(N)W_{p}^{-1} \mat{1}{0}{0}{p} W_p\Gamma_0(N) = 
\Gamma_0(N) \mat{p}{0}{0}{1} \Gamma_0(N)$. We note that
\[W_{p}^{-1} \mat{1}{0}{0}{p} W_p = 
\mat{p\beta-N\gamma}{1-p}{-N\gamma\beta+pN\gamma\beta}
{-\frac{N\gamma}{p}+\beta p^2}.\]

Choose $t \in \Z$ such that $t \equiv \beta M^{-1} \pmod{p}$ and 
consider the matrix $\mat{1}{0}{Nt}{1}$ in $\Gamma_0(N)$. Then,
\[
 W_{p}^{-1} \mat{1}{0}{0}{p} W_p\cdot \mat{1}{0}{Nt}{1}\cdot \mat{p}{0}{0}{1}^{-1}
\]
\[=\mat{p\beta-N\gamma}{1-p}{-N\gamma\beta+pN\gamma\beta}{-\frac{N\gamma}{p}+\beta p^2} \cdot \mat{\frac{1}{p}}{0}{\frac{Nt}{p}}{1}\]
\[=\mat{*}{*}{-N\gamma\left(\frac{\beta- Mt }{p}\right)+N\gamma\beta+\beta p Nt}{*} \in \Gamma_0(N). \]
Hence $W_{p}^{-1} \mat{1}{0}{0}{p} W_p \in \Gamma_0(N) \mat{p}{0}{0}{1} \Gamma_0(N)$ and our claim is proved.

Thus from \eqref{eq:eq1}, we have $f|_{2k}W_{p}^{-1}T_pW_p = f|_{2k}\Gamma_0(N) \mat{p}{0}{0}{1} \Gamma_0(N)=T_p^*(f)$. Using this we get that
\begin{equation*}
\begin{split}
\langle Q_p(f),\ g \rangle &= p^{1-k}\langle T_pW_p(f),\ g \rangle \\
&=p^{1-k}\langle W_p(f),\ T_p^*(g) \rangle \\
&=p^{1-k}\langle W_p(f),\ W_{p}T_pW_p^{-1}(g) \rangle \\
&=p^{1-k}\langle f,\ T_pW_p^{-1}(g) \rangle \\
&=p^{1-k}\langle f,\ T_pW_p(g) \rangle = \langle f,\ Q_p(g) \rangle,
\end{split}
\end{equation*}
since $W_p$ is self-adjoint and it is an involution on $S_{2k}(\Gamma_0(N))$. Hence $Q_p$ and consequently $Q_p'$ are self-adjoint.
\end{proof}
We now restate Theorem~\ref{thm:sec2thm1} and prove it below. 
\begin{thm}\label{thm:pM}
Let $N=p_1p_2\cdots p_r$ with $p_i$ distinct primes. Then 
the space of new forms $S_{2k}^{\mathrm{new}}(\Gamma_0(N))$ is the intersection of
the $-1$ eigenspaces of $Q_{p_i}$ and $Q_{p_i}'$ as $1 \le i \le r$. That is, $f \in S_{2k}^{\mathrm{new}}(\Gamma_0(N))$
if and only if $Q_{p_i}(f)=-f =Q_{p_i}'(f)$ for all $1 \le i \le r$.
\end{thm}
\begin{proof}
We have already seen that if $f\in S_{2k}(\Gamma_0(N))$
then $Q_{p_i}(f)=-f =Q_{p_i}'(f)$ for all $1 \le i \le r$. 

Further it follows from Proposition~\ref{prop:prop1pM} and Lemma~\ref{lem:lem4pm} that for each $p_i$, the subspace
$S_{2k}(\Gamma_0(N/{p_i}))$ is contained in the $p_i$ eigenspace of $Q_{p_i}$ and
$V(p_i)S_{2k}(\Gamma_0(N/{p_i}))$ is contained in the $p_i$ eigenspace of $Q_{p_i}'$.

Suppose $f\in S_{2k}(\Gamma_0(N))$ is such that $Q_{p_i}(f)=-f =Q_{p_i}'(f)$ for all $1 \le i \le r$.
Since $Q_{p_i}$ and $Q_{p_i}'$ are self-adjoint operators on $S_{2k}(\Gamma_0(N))$ we get that
$p_i$ eigenspaces of $Q_{p_i}$ and $Q_{p_i}'$ are respectively orthogonal to $-1$ eigenspaces of 
$Q_{p_i}$ and $Q_{p_i}'$.
Hence
$f$ is orthogonal to $S_{2k}(\Gamma_0(N/{p_i}))$ and $V(p_i)S_{2k}(\Gamma_0(N/{p_i}))$ for all $1 \le i \le r$.
Thus $f$ is orthogonal to the old space $S_{2k}^{\mathrm{old}}(\Gamma_0(N))$, that is $f \in S_{2k}^{\mathrm{new}}(\Gamma_0(N))$.
\end{proof}

\subsection{General case}
Let $N$ be a positive integer and $p$ be a prime such that $p^n$ strictly divides $N$, that
is $N=p^nM$ for some positive integer $M$ where $M$ is coprime to $p$. Let $n \ge 2$. Recall that for
$\mathcal{U}_0,\ \mathcal{U}_n,\ \mathcal{T}_1$ and $\mathcal{Y}_{r} \in H(G//K_0(p^n))$ where $1 \le r \le n$, 
we respectively obtained the classical operators
$Q_{p^n},\ W_{p^n},\ \tilde{U}_p$ and $S_{p^n,r}$ where $\tilde{U}_p$ is as before and
\[W_{p^n}(f) = f|_{2k}\mat{p^n\beta}{1}{N\gamma}{p^n} \quad \text{where $p^{2n}\beta-N\gamma=p^n$},\]
\[Q_{p^n}(f) = (\tilde{U}_p)^n W_{p^n}(f),\]
and
\[S_{p^n,r}(f) = f +  \sum_{j=r}^{n-1}\sum_{s \in \Z_p^*/1+p^{n-j}\Z_p} f|_{2k}A_{s,j}, \
\text{where $A_{s,j}=\mat{a_{s,j}}{b_{s,j}}{p^{j}M}{p^{n-j}-sM}$}\]
is a matrix of determinant $1$. Further we have
\[Q_{p^n}(Q_{p^n}-p^n)(Q_{p^n}+p^{n-1})=0, \qquad S_{p^n,r}(S_{p^n,r}-p^{n-r})=0.\]

We have the following lemma.
\begin{lem}\label{lem:lem2p^nM}
For $1 \le r \le n$, a set of right coset representatives for
$\Gamma_0(N)$ in $\Gamma_0(p^{r}M)$ consists of
the identity element and elements of the form
\[A_{s,j}=\mat{a_{s,j}}{b_{s,j}}{p^{j}M}{p^{n-j}-sM} \ \ \text{where $r\le j\le n-1$ and $s \in\Z_p^*/1+p^{n-j}\Z_p$ }.\]
\end{lem}
\begin{proof}
First we check that the right cosets $\Gamma_0(N)$ and $\Gamma_0(N)A_{s,j}$ where $j$, $s$ varies as above are mutually disjoint.
For any such $j$ and $s$ clearly $A_{s,j} \in \Gamma_0(p^{r}M) \setminus \Gamma_0(N)$, hence
$\Gamma_0(N)A_{s,j}$ and $\Gamma_0(N)$ are disjoint.

Now for any $r \le i,\ j \le n-1$ we have
\[\Gamma_0(N)A_{s,j} = \Gamma_0(N)A_{t,i} \iff p^jM(p^{n-i}-tM)-p^iM(p^{n-j}-sM) \in p^nM\Z_p.\]
Now if $i \ne j$, say $i >j$ then the equality of the above two cosets implies that $-tM \in p\Z_p$ leading to a contradiction.

Similarly, for  $r \le j  \le n-1$ we have
\[\Gamma_0(N)A_{s,j} = \Gamma_0(N)A_{t,j} \iff p^jM (p^{n-j} - tM) - p^jM (p^{n-j} - sM) \in p^nM\Z_p\]
\[\iff t \equiv s \pmod{ p^{n-j} \Z_p} \iff t = s \in \Z_p^*/(1+p^{n-j}\Z_p).\]
Hence all the right cosets listed are mutually disjoint.

It is well known that $[\Gamma_0(p^{r}M) : \Gamma_0(N)] =p^{n-r}$ (~\cite[Theorem 4.2.5]{Miyake}). Since we have already checked
that the right cosets $\Gamma_0(N)$, $\Gamma_0(N)A_{s,j}$ where $j,\ s$ varies as above are mutually disjoint and since there are
exactly $p^{n-r}$ of them the lemma follows.
\end{proof}
\begin{lem}\label{lem:lem3p^nM}
For $1 \le r \le n$, the operator $S_{p^n,r}$ takes the space $S_{2k}(\Gamma_0(N))$ to $S_{2k}(\Gamma_0(p^{r}M))$.
\end{lem}
\begin{proof}
Let $f \in S_{2k}(\Gamma_0(N))$. By the above lemma, the identity element and $A_{s,j}$ for
$r\le j\le n-1$ and $s \in\Z_p^*/1+p^{n-j}\Z_p$ constitute a set of right
coset representatives for $\Gamma_0(N)$ in $\Gamma_0(p^{r}M)$. It follows by \cite[Lemma 3]{A-L} that
\[S_{p^n,r}(f) =  f +  \sum_{j=r}^{n-1}\sum_{s \in \Z_p^*/1+p^{n-j}\Z_p} f|_{2k}A_{s,j} \in S_{2k}(\Gamma_0(p^{r}M)).\]
\end{proof}

\begin{cor}\label{cor:cor1p^nM}
For $1 \le r \le n$, the $p^{n-r}$ eigenspace of $S_{p^n,r}$ is precisely the subspace $S_{2k}(\Gamma_0(p^rM))$.
\end{cor}
\begin{proof}
It follows from Proposition~\ref{prop:trans2} that $S_{2k}(\Gamma_0(p^r M))$ is contained in the 
$p^{n-r}$ eigenspace of $S_{p^n,r}$. 
Let $f \in S_{2k}(\Gamma_0(N))$ be such that $S_{p^n,r}(f) = p^{n-r}(f)$. By Lemma~\ref{lem:lem3p^nM}, $S_{p^n,r}(f)$ belongs to 
$S_{2k}(\Gamma_0(p^{r}M))$. Thus $ f \in S_{2k}(\Gamma_0(p^{r}M))$.
\end{proof}

\begin{prop}\label{prop:prop3p^nM}
Let $1 \le r \le n$. Then for each
$r < \alpha \le n$, the space $S^\mathrm{new}_{2k}(\Gamma_0(p^\alpha M))$ is contained in the $0$ eigenspace of $S_{p^n,r}$.
\end{prop}
\begin{proof}
Let $q$ be any prime that is coprime to $N$, then the 
Hecke operator $T_q$ on $S_{2k}(\Gamma_0(N))$ corresponds to 
$\mathcal{T}_{(q)}$, the characteristic function of the double coset 
$\GL_2(\Z_q)\mat{q}{0}{0}{1}\GL_2(\Z_q)$ in
the $q$-adic Hecke algebra $H(\GL_2(\Z_q))$.
Since $\mathcal{Y}_{r} = {\mathcal{Y}_{r}}_{(p)}$ belongs to the 
$p$-adic Hecke algebra $H(K_0(p^n))$, it follows from Remark~\ref{rem:1} 
that the operators $\mathcal{T}_{(q)}$ and ${\mathcal{Y}_{r}}_{(p)}$ commute 
and hence the operators $S_{p^n,r}$ and $T_q$ on $S_{2k}(\Gamma_0(N))$ commute.

Let $r < \alpha \le n$ and $f \in S^\mathrm{new}_{2k}(\Gamma_0(p^\alpha M))$ be a primitive form. Thus
$f$ is an eigenform with respect to $T_q$ for any $q$ coprime to $N$. Now since $S_{p^n,r}$ and $T_q$ commute we get that
$S_{p^n,r}(f)$ is also an eigenfunction with respect to all such $T_q$ having the same eigenvalue as $f$.

By Corollary~\ref{lem:lem3p^nM}, $S_{p^n,r}(f) \in S_{2k}(\Gamma_0(p^{r}M))$ and as $r < \alpha$, it is an old form in the space
$S_{2k}(\Gamma_0(p^\alpha M))$. It now follows from \cite[Lemma 23]{A-L} that $S_{p^n,r}(f)=0$.

The proposition now follows since $S^\mathrm{new}_{2k}(\Gamma_0(p^\alpha M))$ has a basis of primitive forms.
\end{proof}

Next consider the operator $S_{p^n,r}'=W_{p^n}S_{p^n,r}W_{p^n}^{-1} = W_{p^n}S_{p^n,r}W_{p^n}$. 
Then $S_{p^n,r}'$ clearly satisfies the equation 
$S_{p^n,r}'(S_{p^n,r}'-p^{n-r})=0$.
Since the action of Atkin-Lehner operator $W_{p^n}$ on the space of new forms is surjective, in particular we get that the space
$S^\mathrm{new}_{2k}(\Gamma_0(N))$ is contained in the $0$ eigenspace of $S_{p^n,n-1}'$.
We have the following lemma.
\begin{lem}\label{lem:lem4p^nM}
For $0 \le r \le n$, the operator 
$W_{p^n}$ maps $S_{2k}(\Gamma_0(p^rM))$ onto $V(p^{n-r})S_{2k}(\Gamma_0(p^rM))$ 
and takes the new space $S_{2k}^\mathrm{new}(\Gamma_0(p^rM))$ onto 
$V(p^{n-r})S_{2k}^\mathrm{new}(\Gamma_0(p^rM))$. 

Further,
$W_{p^n}$ maps the space $V(p^r)S_{2k}(\Gamma_0(M))$ onto 
$V(p^{n-r})S_{2k}(\Gamma_0(M))$.

Consequently for $1 \le r \le n$, 
the $p^{n-r}$ eigenspace of $S_{p^n,r}'$ is precisely the space 
$V(p^{n-r})S_{2k}(\Gamma_0(p^rM))$.
\end{lem}
\begin{proof}
Let $r\ge 1$ be as above. Let $f \in S_{2k}(\Gamma_0(p^rM))$. Then, 
\begin{equation*}
\begin{split}
W_{p^n}(f)(z) &= f\left(\frac{p^n\beta z+1}{N\gamma z +p^n}\right)(N\gamma z +p^n)^{-2k}p^{nk} \\
&= f\left(\frac{p^r\beta(p^{n-r}z)+1}{p^rM\gamma (p^{n-r}z) +p^n}\right)(N\gamma z +p^n)^{-2k}p^{nk} \\
&= p^{(n-r)k} f|_{2k} \mat{p^r\beta}{1}{p^rM\gamma}{p^n} (p^{n-r}z) = 
p^{(n-r)k} f|_{2k} W_{p^r} (p^{n-r}z)
\end{split}
\end{equation*}
which clearly belongs to $V(p^{n-r})S_{2k}(\Gamma_0(p^rM))$.

Note that since $W_{p^r}$ is an involution on $S_{2k}(\Gamma_0(p^rM))$, 
it is a surjection, i.e any $f\in S_{2k}(\Gamma_0(p^rM))$ is of the form 
$f'|_{2k} W_{p^r}$ for some $f'\in S_{2k}(\Gamma_0(p^rM))$. 
Let $g(z) = f(p^{n-r} z)$ where $f\in S_{2k}(\Gamma_0(p^rM))$. Then 
by above computation,
\[g(z) = f'|_{2k} W_{p^r}(p^{n-r} z) = p^{(r-n)k} W_{p^n}(f')(z).\]
Thus $W_{p^n}(g)(z) = p^{(r-n)k}(f')(z) = p^{(r-n)k}f|_{2k} W_{p^r}(z)$.

It is clear from above, that if $f \in S_{2k}(M)$ then 
$W_{p^n}(f)(z) = p^{nk}f(p^nz)$ and conversely if $g=f(p^nz)$ then 
$W_{p^n}(g) = p^{-nk}f$, proving the statement for $r=0$.
Moreover Atkin-Lehner involutions $W_{p^r}$ are surjection on new spaces and 
hence takes $S_{2k}^\mathrm{new}(\Gamma_0(p^rM))$ onto 
$V(p^{n-r})S_{2k}^\mathrm{new}(\Gamma_0(p^rM))$.

The proof of the second statement follows similarly. 
For the final statement let $1 \le r \le n$. 
Now $h$ is in the $p^{n-r}$ eigenspace of $S_{p^n,r}'$ 
if and only if $W_{p^n}(h)$ is in the $p^{n-r}$ 
eigenspace of $S_{p^n,r}$. 
By Corollary~\ref{cor:cor1p^nM}, this is same as 
$W_{p^n}(h) \in S_{2k}(\Gamma_0(p^rM))$, that is 
$h \in V(p^{n-r})S_{2k}(\Gamma_0(p^rM))$.
\end{proof}

Applying above results to the case $r =n-1$ we have the following corollary.
\begin{cor}\label{cor:corp^nM}
The space $S_{2k}(\Gamma_0(p^{n-1} M))$ is the $p$ eigenspace of $S_{p^n,n-1}$ and 
$V(p)S_{2k}(\Gamma_0(p^{n-1} M))$ is the $p$ eigenspace of $S_{p^n,n-1}'$. Moreover,
the space $S^\mathrm{new}_{2k}(\Gamma_0(N))$ is contained in the intersection of the
$0$ eigenspaces of $S_{p^n,n-1}$ and $S_{p^n,n-1}'$. 
\end{cor}

Next we have the following proposition.
\begin{prop}\label{prop:prop4p^nM}
The operators $S_{p^n, n-1}$ and $S_{p^n, n-1}'$ are self-adjoint with respect to Petersson inner product.
\end{prop}
\begin{proof}
Since $S_{p^n, n-1} = I + q(\mathcal{V}_{n-1})$, it is enough to prove that $q(\mathcal{V}_{n-1})$ on $S_{2k}(\Gamma_0(N))$ is self-adjoint.
Recall that
\[q(\mathcal{V}_{n-1})(f) =  \sum_{s=1}^{p-1} f |_{2k}A_s \quad \text{where $A_s =\mat{a_s}{b_s}{p^{n-1}M}{p-sM} \in \SL_2(\Z)$}.\]
By \cite[Theorem 2.8.2]{Miyake},
$\langle q(\mathcal{V}_{n-1})(f), g \rangle = 
\langle f, \sum_{s=1}^{p-1} g |_{2k}A_s^{-1} \rangle$.
We claim that for any $f \in S_{2k}(\Gamma_0(N))$ we have
$\sum_{s=1}^{p-1} f |_{2k}A_s = \sum_{t=1}^{p-1} f |_{2k}A_t^{-1}$. Note that for each $1 \le t \le p-1$, the choice of $a_t$ is unique mod $p$.
Let $1 \le s \le p-1$ be such that $s \equiv a_t M^{-1} \pmod{p}$. As $t$ varies from $1$ to $p-1$, so does $s$. Now it is
easy to see that
\[ s \equiv a_t M^{-1} \pmod{p} \iff A_sA_t \in \Gamma_0(N) \iff f |_{2k}A_s = f |_{2k}A_t^{-1}, \]
proving our claim. Thus
\[\langle q(\mathcal{V}_{n-1})(f), g \rangle = \langle f, q(\mathcal{V}_{n-1})g \rangle,\]
and so $S_{p^n,n-1}$ is self-adjoint. Since the Atkin-Lehner operator $W_{p^n}$ is self-adjoint, it follows that $S_{p^n,n-1}'$ is also self-adjoint.
\end{proof}

Now we restate and give a proof of Theorem~\ref{thm:sec2thm4}  
of which Theorem 2' is a particular case. 
\begin{thm}\label{thm:p^nM}
Let $N=p_1p_2\cdots p_rq_1^{\alpha_1}q_2^{\alpha_2}\cdots q_s^{\alpha_s}$ with $p_i$ and $q_j$ distinct primes 
and $\alpha_j \ge 2$ for all $1 \le j \le s$. Then 
the space of new forms $S_{2k}^{\mathrm{new}}(\Gamma_0(N))$ is the intersection of
the $-1$ eigenspaces of $Q_{p_i}$ and $Q_{p_i}'$ as $1 \le i \le r$ and 
$0$ eigenspaces of $S_{{q_j}^{\alpha_j}, \alpha_j-1}$ and 
$S'_{{q_j}^{\alpha_j}, \alpha_j-1}$
for all $1 \le j \le s$. That is, $f \in S_{2k}^{\mathrm{new}}(\Gamma_0(N))$
if and only if $Q_{p_i}(f)=-f =Q_{p_i}'(f)$ for all $1 \le i \le r$ and 
$S_{{q_j}^{\alpha_j}, \alpha_j-1}(f) = 0 =S_{{q_j}^{\alpha_j}, \alpha_j-1}'(f)$ 
for all $1 \le j \le s$.
\end{thm}
\begin{proof}
We have already seen one side implication. 
Conversely suppose $f \in S_{2k}(\Gamma_0(N))$ is such that $Q_{p_i}(f)=-f =Q_{p_i}'(f)$ for 
all $1 \le i \le r$ and $S_{{q_j}^{\alpha_j}, \alpha_j-1}(f) = 0 =S_{{q_j}^{\alpha_j}, \alpha_j-1}'(f)$ 
for all $1 \le j \le s$. It follows from the previous subsection that 
for each $1 \le i \le r$, $S_{2k}(\Gamma_0(N/{p_i}))$ is contained in the $p_i$ eigenspace of $Q_{p_i}$ and
$V(p_i)S_{2k}(\Gamma_0(N/{p_i}))$ is contained in the $p_i$ eigenspace of $Q_{p_i}'$. Also 
from Corollary~\ref{cor:corp^nM}, for each $1 \le j \le s$, we get that 
$S_{2k}(\Gamma_0(N/{q_j}))$ is contained in the $q_j$ eigenspace of $S_{{q_j}^{\alpha_j}, \alpha_j-1}$ and 
$V(q_j)S_{2k}(\Gamma_0(N/{q_j}))$ is contained in the $q_j$ eigenspace of $S_{{q_j}^{\alpha_j}, \alpha_j-1}'$.

Since $Q_{p_i}$, $Q_{p_i}'$ and $S_{{q_j}^{\alpha_j}, \alpha_j-1}$, 
$S_{{q_j}^{\alpha_j}, \alpha_j-1}'$ are self-adjoint operators we get that 
$f$ is orthogonal to $S_{2k}(\Gamma_0(N/{p}))$ and $V(p)S_{2k}(\Gamma_0(N/{p}))$ for each prime divisor $p$ of $N$. 
Thus $f$ is orthogonal to the old space, that is, $f \in S_{2k}^{\mathrm{new}}(\Gamma_0(N))$.
\end{proof}

Next we consider $N$ such that any prime divisor divides it with power at most $2$. 
Let $p$ be a prime such that $N=p^2M$, so $(p,M)=1$.
Recall that
$Q_{p^2} = (\tilde{U}_p)^2 W_{p^2}$ and $Q_{p^2}(Q_{p^2}-p^2)(Q_{p^2}+p)=0$.
It follows from Corollary~\ref{cor:trans3} that if $f \in S_{2k}(\Gamma_0(pM))$ 
then $Q_{p^2}(f)=pQ_p(f)$, hence $Q_{p^2}$ stabilizes $S_{2k}(\Gamma_0(pM))$ 
and acts with eigenvalues $p^2$ and $-p$ on this subspace.
In particular if $f \in S_{2k}(\Gamma_0(M))$ then $Q_{p^2}(f)=p^2f$ and
if $f\in S_{2k}^\mathrm{new}(\Gamma_0(pM))$ then $Q_{p^2}(f)=-pf$.

Finally if $f \in S_{2k}^\mathrm{new}(\Gamma_0(N))$ is a primitive form then 
$\tilde{U}_p(f)=0$ and so $Q_{p^2}(f)=0$.
Thus if $f \in S_{2k}^\mathrm{new}(\Gamma_0(N))$ then $Q_{p^2}(f)=0$.

Consider the operator $Q_{p^2}'=W_{p^2}Q_{p^2}W_{p^2}=W_{p^2}(\tilde{U}_p)^2$, 
then $Q_{p^2}'(Q_{p^2}'-p^2)(Q_{p^2}'+p)=0$. We have the following lemma.
\begin{lem}\label{lem:lem5p^2M}
Let $N=p^2M$ with $(p,M)=1$.
\begin{enumerate} 
\item The operator $Q_{p^2}'$ stabilizes the space
$V(p) S_{2k}(\Gamma_0(pM))$ and its subspace $V(p)X_p$.
\item If $g(z)=f(p^2z) \in V(p^2) S_{2k}(\Gamma_0(M))$ where 
$f \in  S_{2k}(\Gamma_0(M))$, then $Q_{p^2}'(g)=p^2 g$. Consequently, 
$Q_{p^2}'$ has eigenvalues $p^2$ and $-p$ on the space $V(p)X_p$. 
\item If $f \in S_{2k}^\mathrm{new}(\Gamma_0(pM))$ and 
$g=f(pz) \in V(p)S_{2k}^\mathrm{new}(\Gamma_0(pM))$. 
Then $Q_{p^2}'(g)=-p g$.
\item Let $q$, $M'$ be positive integers such that $(q,p)=1$ and $qM' \mid M$. 
Then $V(pq)S_{2k}^\mathrm{new}(\Gamma_0(pM'))$ is contained in 
the $-p$ eigenspace of $Q_{p^2}'$.  
\end{enumerate}
Thus $Q_{p^2}'$ acts with eigenvalues $p^2$ and $-p$ on $V(p) S_{2k}(\Gamma_0(pM))$.
\end{lem}
\begin{proof}
Let $g = f(pz)$ where $f \in S_{2k}(\Gamma_0(pM))$. It follows from
Lemma~\ref{lem:lem4p^nM} that $W_{p^2}(g) = p^{-k}W_p(f)$ where $W_p$ acts
via $\mat{p^2\beta}{1}{pM\gamma}{p}$. Since $W_p$ is Atkin-Lehner operator on 
$S_{2k}(\Gamma_0(pM))$ and $Q_{p^2}$ stabilizes $S_{2k}(\Gamma_0(pM))$ and 
$W_p^2$ maps $S_{2k}(\Gamma_0(pM))$ onto $V(p)S_{2k}(\Gamma_0(pM))$
we get that $Q_{p^2}'(g)$ belongs to $V(p)S_{2k}(\Gamma_0(pM))$. 
Thus $Q_{p^2}'$ stabilizes $V(p) S_{2k}(\Gamma_0(pM))$.

In particular, if $f \in S_{2k}^\mathrm{new}(\Gamma_0(pM))$, since 
$W_p$ preserves the space of newforms, we get that
$W_{p^2}(g)$ belongs to $S_{2k}^\mathrm{new}(\Gamma_0(pM))$. Thus
\[Q_{p^2}'(g) = W_{p^2}Q_{p^2}(W_{p^2}(g)) = -pW_{p^2}(W_{p^2}(g)) = -pg, \]
proving $(3)$. 

Recall that $V(p)X_p= V(p)S_{2k}(\Gamma_0(M)) \oplus V(p^2)S_{2k}(\Gamma_0(M))$. 
Let $g(z)=f(p^2z)$ where $f \in  S_{2k}(\Gamma_0(M))$, then by
Lemma~\ref{lem:lem4p^nM}, we get that $W_{p^2}(g) = p^{-2k}f$ and thus 
\[Q_{p^2}'(g) = W_{p^2}Q_{p^2}(p^{-2k}f) = p^2 W_{p^2}(p^{-2k}f)=p^2g, \]
proving part of $(2)$. Now we shall complete proof of $(1)$ and $(2)$. 

Let $g(z)=f(pz)$ where $f \in  S_{2k}(\Gamma_0(M))$. By Lemma~\ref{lem:lem4p^nM}, 
$W_{p^2}(g)=g$ and using Lemma~\ref{lem:lem3pM} we get
\[Q_{p^2}'(g) = W_{p^2}Q_{p^2}(g) = 
pW_{p^2}(p^{1-2k}T_p(f)-g)=p^2T_p(f)(p^2z)-pg,\]
which clearly belongs to $V(p)X_p$, showing $(1)$.
Now following arguments as in Corollary~\ref{cor:eigspcpM} and Proposition~\ref{prop:prop1pM}, we get that 
$Q_{p^2}'$ acts with eigenvalues $p^2$ and $-p$ on $V(p)X_p$ and the $p^2$ eigenspace of $Q_{p^2}'$ inside $V(p)X_p$ is 
$V(p^2)S_{2k}(\Gamma_0(M))$.

To prove $(4)$, we check that the operators $V(q)$ and $Q_{p^2}'$ commutes on $S_{2k}(\Gamma_0(p^2M'))$. Since $(\tilde{U}_p)^2$ 
commutes with $V(q)$~\cite[Lemma 15]{A-L} enough to check that $W_{p^2}$ commutes with $V(q)$. Let $W_{p^2}$ acts via 
$\mat{p^2\beta}{1}{N\gamma}{p^2}$ of determinant $p^2$, then
$\mat{q}{0}{0}{1}W_{p^2}\left(W_{p^2}\mat{q}{0}{0}{1}\right)^{-1}$ 
belongs to $\Gamma_0(N/q)$. So for 
$f \in S_{2k}(\Gamma_0(N/q))$, $W_p^2 V(q) (f) = V(q)W_p^2 (f)$.
Hence 
$Q_{p^2}'V(pq)(f) = V(q)Q_{p^2}'V(p)(f)$ for 
$f \in S_{2k}^\mathrm{new}(\Gamma_0(pM'))$.

We can check that $V(p)S_{2k}^\mathrm{new}(\Gamma_0(pM'))$ is contained in 
the $-p$ eigenspace of $Q_{p^2}'$ and so, 
$Q_{p^2}'V(pq)S_{2k}^\mathrm{new}(\Gamma_0(pM')) = -pV(q)V(p)S_{2k}^\mathrm{new}(\Gamma_0(pM'))$ concluding the proof.

Finally since \[V(p)S_{2k}(\Gamma_0(pM)) = V(p)S_{2k}^\mathrm{new}(\Gamma_0(pM)) \oplus V(p)X_p \oplus\]
\[\oplus_{qM'\mid M, (q,p)=1} V(pq)S_{2k}^\mathrm{new}(\Gamma_0(pM')),\]
we get that $Q_{p^2}'$ acts with eigenvalues $p^2$ and $-p$ on $V(p) S_{2k}(\Gamma_0(pM))$.
\end{proof}
\begin{prop}\label{prop:prop5p^nM}
The operators $Q_{p^2}=(\tilde{U}_p)^2 W_{p^2}$ and $Q_{p^2}'=W_{p^2}Q_{p^2}W_{p^2}$ are self-adjoint with respect to Petersson inner product.
\end{prop}
\begin{proof}
The proof is similar to that of Proposition~\ref{prop:prop2pM}.
\end{proof}
Now we restate and prove Theorem~\ref{thm:sec2thm2}.
\begin{thm}\label{thm:p^2M}
Let $N=M_1^2M$ where $M_1,\ M$ are square free and coprime. Then $f \in S_{2k}^{\mathrm{new}}(\Gamma_0(N))$
if and only if $Q_{p}(f)=-f =Q_{p}'(f)$ for all primes $p$ dividing $M$ and 
$Q_{p^2}(f)=0 =Q_{p^2}'(f)$ for all primes $p$ dividing $M_1$.
\end{thm}
\begin{proof}
The one side implication is clear. 

Conversely if $f \in S_{2k}(\Gamma_0(N))$ is such that $Q_{p}(f)=-f =Q_{p}'(f)$ for all primes $p \mid M$, then as before 
$f$ is orthogonal to both $S_{2k}(\Gamma_0(N/{p}))$ and $V(p)S_{2k}(\Gamma_0(N/{p}))$ for all $p \mid M$. 

Let $q$ be a prime dividing $M_1$ and $N=q^2N'$, so $(q,N')=1$. We have already checked that $Q_{q^2}'$ stabilizes 
$S_{2k}(\Gamma_0(N/{q})) = S_{2k}(\Gamma_0(qN'))$ and acts with eigenvalues $q^2$ and $-q$. Further it follows from 
Lemma~\ref{lem:lem5p^2M} that $Q_{q^2}'$ stabilizes 
$V(q)S_{2k}(\Gamma_0(N/{q}))$ i.e, $V(q)S_{2k}(\Gamma_0(qN'))$ and acts with eigenvalues $q^2$ and $-q$.
Thus if $Q_{q^2}(f)=0 =Q_{q^2}'(f)$ for all primes $q$ dividing $M_1$ we get that 
$f$ is orthogonal to both $S_{2k}(\Gamma_0(N/{q}))$ and $V(q)S_{2k}(\Gamma_0(N/{q}))$ for all $q \mid M_1$. 
Hence $f$ is in the new space at level $N$. 
\end{proof}

Let $p$ be an odd prime. Next we shall consider the action of twisting operators $R_{p}$ and $R_{\chi}$ \cite[Section 6]{A-L}
where $R_p$ is the twist by the Dirichlet character given by Kronecker symbol $\kro{\cdot}{p}$ and
$R_{\chi}$ is the twist by the Dirichlet character given by $\kro{-1}{\cdot}$. To be more precise, let
$f(z) = \sum_{n=1}^{\infty}a_nq^n \in S_{2k}(\Gamma_0(N))$. Then
\[R_p (f)(z) = \sum_{n=1}^{\infty} \kro{n}{p}a_nq^n, \qquad R_{\chi} (f)(z) = \sum_{n=1}^{\infty} \kro{-1}{n}a_nq^n.\]
By \cite[Lemma 33]{A-L}, $R_p$ and $R_{\chi}$ are operators on $S_{2k}(\Gamma_0(N))$ provided that $p^2 \mid N$ and $16 \mid N$
respectively. 

It is well known that 
$R_p$ and  $R_{\chi}$ are self-adjoint operators with respect to 
Petersson inner product.
\begin{lem}\label{lem:lem8p^nM}
Let $N=p^nM$ where $p$ is odd and coprime to $M$ and $n \ge 2$.
If $f \in S_{2k}^\mathrm{new}(\Gamma_0(N))$, then $(R_p)^2(f)=f$.
For $1 \le \alpha \le n$ the space $V(p^\alpha)(S_{2k}(\Gamma_0(p^{n-\alpha}M)))$
is contained in the $0$ eigenspace of $R_p^2$.
\end{lem}
\begin{proof}
If $f(z) = \sum_{n=1}^{\infty}a_nq^n \in S_{2k}^\mathrm{new}(\Gamma_0(N))$ is a primitive form, as $p^2 \mid N$, we have
$a_p=0$ and consequently $a_m=0$ for any $m$ divisible by $p$. Thus
$f(z) = \sum_{\substack{n=1\\(n,p)=1}}^{\infty}a_nq^n$.
Since $S_{2k}^\mathrm{new}(\Gamma_0(N))$ has a basis of primitive forms, this holds for any $f \in S_{2k}^\mathrm{new}(\Gamma_0(N))$.
It now follows that
\[R_p^2(f)(z) = \sum_{\substack{n=1\\(n,p)=1}}^{\infty} \kro{n^2}{p}a_nq^n = \sum_{\substack{n=1\\(n,p)=1}}^{\infty}a_nq^n =f(z).\]
Let $g(z) = f(p^\alpha z)$ where
$f(z) = \sum_{n=1}^{\infty}a_nq^n  \in S_{2k}(\Gamma_0(p^{n-\alpha}M))$. Then $g(z) = \sum_{n=1}^{\infty}a_nq^{p^{\alpha}n}$.
Since $\alpha \ge 1$, clearly $R_p(g) =0$. Hence the lemma follows.
\end{proof}
Following exactly similar arguments we also have the following lemma.
\begin{lem}\label{lem:lem9p^nM}
Let $N=2^nM$ with $M$ odd and $n \ge 4$.
If $f \in S_{2k}^\mathrm{new}(\Gamma_0(N))$, then $(R_{\chi})^2(f)=f$.
For $1 \le \alpha \le n$ the space $V(p^\alpha)(S_{2k}(\Gamma_0(p^{n-\alpha}M)))$
is contained in the $0$ eigenspace of $R_{\chi}^2$.
\end{lem}

Since $R_p^2$ and $R_{\chi}^2$ are self-adjoint operators, 
using Corollary~\ref{cor:corp^nM} and
Lemmas~\ref{lem:lem8p^nM} and \ref{lem:lem9p^nM},
and following a similar argument as in Theorem~\ref{thm:p^nM} we obtain 
the following theorem (Theorem~\ref{thm:sec2thm5} of Section~\ref{sec:mainresults}).
\begin{thm}
Let $N=2^{\beta}p_1p_2\cdots p_rq_1^{\alpha_1}q_2^{\alpha_2}\cdots q_s^{\alpha_s}$ 
where $p_i,\ q_i$ are distinct odd primes and $\beta \ge 4$ and $\alpha_j \ge 2$ 
for all $1\le j \le s$. Then 
$f \in S_{2k}^{\mathrm{new}}(\Gamma_0(N))$ 
if and only if $Q_{p_i}(f) = -f = Q_{p_i}'(f)$ for all $1\le i \le r$, 
$(R_{q_j})^2(f) = f$ for all $1\le j \le s$ and $(R_{\chi})^2(f) = f$, and 
$S_{q^\gamma,\gamma-1}(f)=0$ for all primes $q$ such that 
$q^\gamma \| N$ with $\gamma \ge 2$.
\end{thm}

\section{Characterization of old spaces}

In the previous section we described the space of newforms in 
$S_{2k}(\Gamma_0(N))$ as a common eigenspace of certain Hecke operators.
In this section we extend this description to the subspaces of old forms 
of type $V(d)S_{2k}^{\mathrm{new}}(\Gamma_0(M))$ that appear in the 
direct sum decomposition of the old space 
$S_{2k}^{\mathrm{old}}(\Gamma_0(N))$ in \eqref{eq:1}. 

We first consider the case when $N$ is square-free. In the theorem below 
we characterize the various summands in the old space  as 
common eigenspaces of the operators $Q_p$, $Q_p'$ as $p$ varies over prime 
divisors of $N$.

\begin{thm}\label{thm:pMsec}
Let $N$ be square-free. Then 
\begin{enumerate}
 \item $f \in S_{2k}(\Gamma_0(1))$ if and only if $Q_p(f) = pf$ for all $p\mid N$.
 \item Let $1 \ne M \mid N$. Then 
 $f \in S_{2k}^{\mathrm{new}}(\Gamma_0(M))$ if and only if $Q_p(f) = -f = Q_p'f$ for all $p\mid M$ and
 $Q_q(f) = qf$ for all $p\mid (N/M)$.
 \item Let $1 \ne M' \mid N$. Then $f \in V(M')S_{2k}(\Gamma_0(1))$ if and only if 
 $Q_q'(f) = qf$ for all $q\mid M'$ and
 $Q_q(f) = qf$ for all $q\mid (N/M')$.
 \item Let $M$ and $M' >1$ and $MM' \mid N$. Then 
 $f \in V(M')S_{2k}^{\mathrm{new}}(\Gamma_0(M))$ if and only if
 $Q_p(f) = -f = Q_p'f$ for all $p\mid M$,
 $Q_q'(f) = qf$ for all $q\mid M'$ and
 $Q_q(f) = qf$ for all $q\mid (N/MM')$.
\end{enumerate}
\end{thm}
The proof relies on the above description of 
eigenspaces of $Q_p$ and $Q_p'$ and the following 
additional lemma. 

\begin{lem}
Let $dM \mid N$ where $M \ne 1$ and $d$ is coprime to $M$. If 
$f \in V(d)S_{2k}^{\mathrm{new}}(\Gamma_0(M))$, then 
$Q_p(f) = -f = Q_p'f$ for all $p\mid M$.
\end{lem}
\begin{proof}
Let $f = V(d)f_1$ where $f_1 \in S_{2k}^{\mathrm{new}}(\Gamma_0(M))$ and 
$p$ be a prime divisor of $M$.
Then $Q_p(f) = \tilde{U}_p W_{p,N} (V(d)f_1)$ where $W_{p,N}$ is the Atkin-Lehner 
operator on $S_{2k}(\Gamma_0(N))$.
Note that for $f \in S_{2k}(M)$, we have $W_{p,N}(f) = W_{p,M}(f)$. 
Further $W_{p,N}$ commutes with $V(d)$ on $S_{2k}(\Gamma_0(M))$ as 
the matrix 
$W_{p,N}V(d) (V(d)W_{p,N})^{-1} \in \Gamma_0(M)$.
Now by~\cite[Lemma 15]{A-L}, $\tilde{U}_p$ commutes with $V(d)$ as $(d,p)=1$. 
Hence by Theorem~\ref{thm:pM},
\[Q_p(f) = V(d)\tilde{U}_p W_{p,M}f_1= V(d)Q_p(f_1) = -V(d)f_1 = -f.\]
The case of $Q_p'$ follows similarly.
\end{proof}

\begin{proof}[Proof of Theorem~\ref{thm:pMsec}]
We shall give proof of $(4)$. The other parts follow similarly.
Let $M$ and $M' >1$ and $N=MM't$ for some $t\in \N$.
If $f\in V(M')S_{2k}^{\mathrm{new}}(\Gamma_0(M))$, then by above lemma 
$Q_p(f) = -f = Q_p'f$ for all $p\mid M$. Further for each $q \mid M'$, 
$V(M')S_{2k}^{\mathrm{new}}(\Gamma_0(M)) \subseteq V(q)S_{2k}(\Gamma_0(N/q))$, 
and so $Q_q'(f) = qf$ for all $q \mid M'$. Similarly for each $q \mid t$ we have 
$V(M')S_{2k}^{\mathrm{new}}(\Gamma_0(M)) \subseteq S_{2k}(\Gamma_0(N/q))$ and so 
$Q_q(f) = qf$ for all $q \mid t$.

Conversely let $f \in S_{2k}(\Gamma_0(N))$ be such that 
$Q_p(f) = -f = Q_p'f$ for all $p\mid M$, $Q_q'(f) = qf$ for all $q\mid M'$ and 
$Q_q(f) = qf$ for all $q\mid (N/MM')$. Let $q$ be any prime such that $q\mid M't$.
Let $V:=\oplus_{dr\mid N,q\mid r} V(d)S_{2k}^{\mathrm{new}}(\Gamma_0(r))$ and  
$W:=\oplus_{dr\mid N,(q,r)=1} V(d)S_{2k}^{\mathrm{new}}(\Gamma_0(r))$
Then $S_{2k}^{\mathrm{old}}(\Gamma_0(N)) =V \oplus W$ and since $N$ is 
square-free we have $W=X_q$. By previous lemma $V$ is contained in the intersection of 
$-1$ eigenspace of $Q_q$ and $Q_q'$. Now $f$ can be uniquely written as 
$f = v + w$ with $v\in V$ and $w\in W$. If $q\mid t$, then $Q_qf = qf$ and 
so $qv + qw = Q_qv + Q_qw = -v + Q_qw$ where $Q_qw \in W$. Thus $v=0$ and 
$f\in W$. If $q\mid M'$ we get the same conclusion by using the operator 
$Q_q'$ instead. Since the above argument works for all primes dividing $M't$, 
we get that 
$ f \in \oplus_{dr \mid N, r \mid M} V(d)S_{2k}^{\mathrm{new}}(\Gamma_0(r))$.

Now let $q \mid M'$ be any prime. Then 
$\oplus_{dr \mid N, r \mid M, (d,q)=1} V(d)S_{2k}^{\mathrm{new}}(\Gamma_0(r)) 
\subseteq S_{2k}(\Gamma_0(N/q))$ while 
$\oplus_{dr \mid N, r \mid M, q \mid d} V(d)S_{2k}^{\mathrm{new}}(\Gamma_0(r)) 
\subseteq V(q)S_{2k}(\Gamma_0(N/q))$. Thus $f \in X_q$. Since $Q_q'f = qf$ and 
the $q$ eigenspace of $Q_q'$ in $X_q$ is precisely 
$V(q)S_{2k}(\Gamma_0(N/q))$, we get that $f$ belongs to 
$\oplus_{dr \mid N, r \mid M, q\mid d}
V(d)S_{2k}^{\mathrm{new}}(\Gamma_0(r))$. 
Applying the same argument for all primes $q\mid M'$ we get that 
$f$ belongs to 
$\oplus_{dr \mid N, r \mid M, M'\mid d} V(d)S_{2k}^{\mathrm{new}}(\Gamma_0(r))$.
Now let $q$ be a prime dividing $t$. Then we have that
$\oplus_{dr \mid N, r \mid M, M'\mid d, (d,q)=1}
V(d)S_{2k}^{\mathrm{new}}(\Gamma_0(r)) \subseteq S_{2k}(\Gamma_0(N/q))$ while 
$\oplus_{dr \mid N, r \mid M, M'\mid d, q \mid d} 
V(d)S_{2k}^{\mathrm{new}}(\Gamma_0(r)) \subseteq V(q)S_{2k}(\Gamma_0(N/q))$. 
Thus $f \in X_q$. Now $Q_qf = qf$ implies that $f\in \oplus_{dr \mid N, r \mid M, M'\mid d, (d,q)=1}
V(d)S_{2k}^{\mathrm{new}}(\Gamma_0(r))$. As before applying this argument for
all primes $q \mid t$ we get 
that $f$ belongs to $\oplus_{dr \mid MM', r \mid M, M'\mid d} 
V(d)S_{2k}^{\mathrm{new}}(\Gamma_0(r)):=Y$. 

Finally let $p$ be a prime dividing $M$. Then
$Y = Y_1 \oplus Y_2 \oplus Y_3$ where 
$Y_1 = \oplus_{dr \mid MM', r \mid M, M'\mid d, (dr,p)=1} 
V(d)S_{2k}^{\mathrm{new}}(\Gamma_0(r))$, 
$Y_2 = \oplus_{dr \mid MM', r \mid M, M'\mid d, p\mid d}\newline
V(d)S_{2k}^{\mathrm{new}}(\Gamma_0(r))$ and 
$Y_3 = \oplus_{dr \mid MM', r \mid M, M'\mid d, p\mid r} 
V(d)S_{2k}^{\mathrm{new}}(\Gamma_0(r))$. Clearly 
$Y_1 \oplus Y_2 \subseteq X_p$. We write $f$ uniquely as 
$f = g + h$ where $g\in Y_1 \oplus Y_2$ and $h \in Y_3$. 
Since $Q_p(f)=-f=Q_p'f$ and $Q_p(h)=-h=Q_p'h$ we get that 
$Q_p(g)=-g=Q_p'g$. Thus $g$ is orthogonal to $X_p$ but 
$g \in X_p$, hence $g=0$. Applying the same argument for 
all primes $p$ dividing $M$ we get that 
$f \in \oplus_{dr \mid MM', r=M, M'\mid d}$ which is precisely
$V(M')S_{2k}^{\mathrm{new}}(\Gamma_0(M))$.
\end{proof}

We now consider the case $N=p^n$ where $p$ is a prime. The characterization
of the old space summands will be done inductively on $n$. 
The case $n=1$ follows from Theorem~\ref{thm:pMsec}. We assume that $n \ge 2$. 
It follows from \eqref{eq:1} that 
\[S_{2k}(\Gamma_0(p^{n})) = S_{2k}(\Gamma_0(p^{n-1})) \oplus 
\bigoplus_{r=0}^n V(p^{n-r})S_{2k}^{\mathrm{new}}(\Gamma_0(p^r)).\]
By Corollary~\ref{cor:cor1p^nM}, 
$S_{2k}(\Gamma_0(p^{n-1}))$ is precisely the $p$ eigenspace 
of the operator $S_{p^n,p^{n-1}}$ on $S_{2k}(\Gamma_0(p^{n}))$ and hence we 
can characterize the summands that appear inside the direct sum 
decomposition of $S_{2k}(\Gamma_0(p^{n-1}))$ using induction hypothesis. 

So we need to only deal with the spaces of type 
$V(p^{n-r})S_{2k}^{\mathrm{new}}(\Gamma_0(p^r))$ for $0 \le r \le n$. 
Using Lemma~\ref{lem:lem4p^nM} the operator 
$W_{p^n}$ maps $S_{2k}^\mathrm{new}(\Gamma_0(p^r))$ onto 
$V(p^{n-r})S_{2k}^\mathrm{new}(\Gamma_0(p^r))$. Thus a 
form $f \in S_{2k}(\Gamma_0(p^{n}))$ belongs to the space
$V(p^{n-r})S_{2k}^{\mathrm{new}}(\Gamma_0(p^r))$ if and only if 
$W_{p^n}(f)$ belongs to $S_{2k}^{\mathrm{new}}(\Gamma_0(p^r))$. 
By the previous section we already know how to characterize the 
forms in $S_{2k}^{\mathrm{new}}(\Gamma_0(p^r))$, thus we 
can characterize $W_{p^n}(f)$ and hence $f$. 

Using above similar statement as Theorem~\ref{thm:pMsec} can be made 
for a general level $N$.

\end{document}